\documentclass{article}
\usepackage{amsmath}
\usepackage{amsfonts,amssymb,mathrsfs}
\usepackage{theorem}
\usepackage{hyperref}

\usepackage{tikz}
\usepackage{tikz-cd}
\usepgflibrary{arrows}
\usetikzlibrary{matrix}
\tikzset{>=stealth'} 

\newcommand{\Ob}{\mathrm{Ob}}

\newcommand{\Core}{\mathrm{Core}}
\newcommand{\lcm}{\mathrm{lcm}}
\newcommand{\Ih}{\mathrm{Ih}}

\newcommand{\into}{\hookrightarrow}
\newcommand{\tto}{\longrightarrow}
\newcommand{\hence}{\Rightarrow}
\newcommand{\iso}{\overset{\simeq}{\tto}}

\input amssym.def

\numberwithin{equation}{section}

\numberwithin{figure}{section}

\makeatletter
\newcommand\qedsymbol{\hbox{$\Box$}}
\newcommand\qed{\relax\ifmmode\Box\else
  {\unskip\nobreak\hfil\penalty50\hskip1em\null\nobreak\hfil\qedsymbol
  \parfillskip=\z@\finalhyphendemerits=0\endgraf}\fi}

\newenvironment{proof-of}[1][{}]{\par\noindent \textbf{Proof of} {#1}. }{\qed}

\newenvironment{proof}[0]{\par\noindent \textbf{Proof}.}{\qed}

\DeclareMathOperator*{\Gal}{Gal}

\newcommand{\GT}{\mathsf{GT}}
\newcommand{\Dih}{\mathsf{Dih}}
\newcommand{\GTSh}{\mathsf{GTSh}}
\newcommand{\GTh}{\widehat{\mathsf{GT}}}
\newcommand{\Zhat}{\widehat{\mathbb{Z}}}

\newcommand{\Fh}{\widehat{\mathrm{F}}}

\newcommand{\PaB}{\mathsf{PaB}}
\newcommand{\NFI}{\mathsf{NFI}}

\newcommand{\F}{\mathrm{F}} 
\newcommand{\PB}{\mathrm{PB}} 
\newcommand{\B}{\mathrm{B}} 
\renewcommand{\H}{\mathsf{H}} 
\newcommand{\K}{\mathsf{K}} 
\renewcommand{\L}{\mathsf{L}} 
\newcommand{\N}{\mathsf{N}} 

\newcommand{\PR}{\mathscr{PR}}

\newcommand{\End}{{\mathsf{End}}}

\newcommand{\Aut}{\mathrm{Aut}}

\newcommand{\ord}{\mathrm{ord}}
\newcommand{\isom}{\mathrm{isom}}

\newcommand{\conn}{\mathrm{conn}}

\newcommand{\id}{\mathrm{id}}

\newcommand{\ti}[1]{{\tilde{#1}}}

\newcommand{\wh}[1]{{\widehat{#1}}}
\newcommand{\ol}[1]{{\overline{#1}}}

\newcommand{\e}[1]{{\textbf{#1}}}
\newcommand{\txt}[1]{{\textrm{#1}}}

\newcommand{\dia}{\diamond}

\newcommand{\hs}{\heartsuit}

\newcommand{\la}{{\lambda}}

\newcommand{\si}{{\sigma}}
\newcommand{\ga}{{\gamma}}
\newcommand{\vf}{{\varphi}}

\newcommand{\ka}{{\kappa}}

\newcommand{\ML}{{\mathcal{ML}}}

\newcommand{\cF}{\mathcal{F}}

\newcommand{\cC}{\mathcal{C}}

\newcommand{\cZ}{\mathcal{Z}}

\newcommand{\cP}{\mathcal{P}}
\newcommand{\cR}{\mathcal{R}}

\newcommand{\hcP}{\widehat{\mathcal{P}}}

\newcommand{\ZZ}{{\mathbb Z}}
\newcommand{\QQ}{{\mathbb Q}}

\newcommand{\te}{\theta}
\newcommand{\Te}{\Theta}

\newcommand{\D}{{\Delta}}

\renewcommand{\mod}{~\mathrm{mod}~}

\newcommand{\lan}{\langle\,}
\newcommand{\ran}{\,\rangle}

\date{}
\newtheorem{thm}{Theorem}[section]
\newtheorem{defi}[thm]{Definition}

\newtheorem{lem}[thm]{Lemma}
\newtheorem{cor}[thm]{Corollary}

\newtheorem{prop}[thm]{Proposition}

\theorembodyfont{\rm}
\newtheorem{remark}[thm]{Remark}

\title{$\GT$-shadows for the gentle version $\GTh_{gen}$ of the Grothendieck-Teichmueller group}

\author{Vasily A. Dolgushev and Jacob J. Guynee}

\date{}

\begin{document}


\maketitle

\begin{abstract}
Let $\B_3$ be the Artin braid group on $3$ strands and $\PB_3$ be the corresponding pure braid group. 
In this paper, we construct the groupoid $\GTSh$ of $\GT$-shadows for a (possibly more tractable) 
version $\GTh_0$ of the Grothendieck-Teichmueller group $\GTh$ introduced in paper 
\cite{HS-fund-groups} by D. Harbater and L. Schneps. We call this group the gentle version of 
$\GTh$ and denote it by $\GTh_{gen}$. The objects of $\GTSh$ are finite index normal subgroups 
$\N$ of $\B_3$ satisfying the condition $\N \le \PB_3$. Morphisms of $\GTSh$ are called $\GT$-shadows 
and they may be thought of as approximations to elements of $\GTh_{gen}$. 
We show how $\GT$-shadows can be obtained from elements of $\GTh_{gen}$ and prove 
that $\GTh_{gen}$ is isomorphic to the limit of a certain functor defined in terms of the groupoid $\GTSh$. 
Using this result, we get a criterion for identifying genuine $\GT$-shadows.
\end{abstract}

\section{Introduction}
\label{sec:intro}

Let $\Fh_2$ be the profinite completion of the free group $\F_2 := \lan x, y \ran$ on two generators
and $\Zhat$ be the profinite completion of the ring of integers. 
The profinite version $\GTh$ of the Grothendieck-Teichmueller group \cite[Section 4]{Drinfeld}
is one of the most mysterious objects in mathematics. It consists of pairs 
$(\hat{m}, \hat{f}) \in \Zhat \times \Fh_2$ satisfying the hexagon relations: 
\begin{equation}
\label{hexa1-intro}
\si_1^{2 \hat{m} + 1}   \hat{f}^{-1} \si_2^{2 \hat{m} + 1}  \hat{f} =   
\hat{f}^{-1} \si_1 \si_2\, \si_1^{-2\hat{m}} c^{\hat{m}}\,,
\end{equation}
\begin{equation}
\label{hexa11-intro}
\hat{f}^{-1} \si_2^{2\hat{m}+1} \hat{f} \, \si_1^{2\hat{m}+1}  = \si_2 \si_1 \si_{2}^{-2\hat{m}} c^{\hat{m}} \, \hat{f} \,,
\end{equation}
the pentagon relation:
\begin{equation}
\label{penta-intro}
\hat{f}(x_{23}, x_{34}) \hat{f}(x_{12} x_{13}, x_{24} x_{34}) \hat{f}(x_{12}, x_{23})  ~ = ~ 
\hat{f}(x_{12}, x_{23} x_{24}) \hat{f}(x_{13} x_{23} ,x_{34})
\end{equation}
and the invertibility condition. In relations \eqref{hexa1-intro} and \eqref{hexa11-intro}, $\si_1$ and $\si_2$
are the standard generators of the Artin braid group $\B_3$, $c := (\si_1 \si_2 \si_1)^2$ and 
$\wh{\F}_2$ is considered as the subgroup of $\wh{\B}_3$; namely, $\wh{\F}_2$ is identified 
with the profinite completion of the subgroup $\lan \si^2_1, \si^2_2 \ran$ of $\B_3$. 

In relation \eqref{penta-intro}, $x_{12} := \si_1^2,\, x_{23}:= \si_2^2,\,\dots$ are the standard generators 
\cite[Section 1.3]{Braids} 
of the pure braid group 
$\PB_4$ on $4$ strands and $\hat{f}(x_{23}, x_{34}), \hat{f}(x_{12} x_{13}, x_{24} x_{34}),  \hat{f}(x_{12}, x_{23}), \dots$ 
are the images of $\hat{f}$ with respect to natural (continuous) group homomorphisms $\wh{\F}_2 \to \wh{\PB}_4$, e.g. 
$\hat{f}(x_{12} x_{13}, x_{24} x_{34})$ is the image of the continuous group homomorphism 
$\wh{\F}_2 \to \wh{\PB}_4$ which sends $x$ (resp. $y$) to $x_{12} x_{13} \in \PB_4 \le \wh{\PB}_4$ 
(resp. $ x_{24} x_{34} \in \PB_4 \le \wh{\PB}_4$). 
 
The multiplication on $\GTh$ can be defined by an explicit formula 
(see equations \eqref{F2-to-F2-hat}, \eqref{bullet} or \cite[Section 1.1]{Leila-survey}) 
or by identifying elements of $\GTh$ with continuous automorphisms of $\wh{\F}_2$
(see \cite[Introduction]{HS-fund-groups}). 

The group $\GTh$ and its variants are a part of an active 
area of research\footnote{The lists of references in this paragraph are 
far from complete.} \cite{OpenGTLeila}, \cite{SashaLeilaCohomological},  
\cite{MN-Aut-Bn-hat}, \cite{Pop-tripod}, \cite{Pop-survey}, \cite{Pop-Topaz}, 
\cite{Leila-survey} and this research is often motivated by fruitful links between operads, 
moduli of curves and the geometric action of the absolute Galois group $G_{\QQ}$ of the field of rational 
numbers \cite{BB-Horel-Robertson}, \cite{Combe-Kalugin}, \cite{Combe-Manin}, \cite{Jordan},
\cite{HS-fund-groups}, \cite{HSL-T-tower}, \cite{HMM-numer-invar}, \cite{LNS-new-version-GT}, 
\cite{NS-subgroup-of-GT}. 

In paper \cite{GTshadows}, the authors constructed an infinite groupoid closely 
related to the group $\GTh$. The objects of this groupoid are finite index
normal subgroups $\N$ of the Artin braid group $\B_4$ satisfying the condition 
$\N \le \PB_4$. The morphisms of this groupoid are called $\GT$-\e{shadows}. 
In addition to other things, the authors of \cite{GTshadows}
proved that $\GTh$ is isomorphic to the limit of a certain functor defined in terms of the groupoid
of $\GT$-shadows (see \cite[Section 3]{GTshadows}). In this respect, certain $\GT$-shadows are approximations 
of elements of the group $\GTh$.   

The purpose of this paper is to develop a version of the groupoid of $\GT$-shadows 
for the gentle version $\GTh_{gen}$ of the Grothendieck-Teichmueller group $\GTh$. 
Just as $\GTh$, the group $\GTh_{gen}$ consists of pairs $(\hat{m}, \hat{f}) \in \Zhat \times \Fh_2$
satisfying hexagon relations \eqref{hexa1-intro}, \eqref{hexa11-intro}, the invertibility 
condition and the following consequence of pentagon relation \eqref{penta-intro}: 
$$
\hat{f} \in [\Fh_2, \Fh_2]^{top.\,cl.}\,.
$$
For more details, please see Subsection \ref{sec:GTh-gen-mon-GTh-gen}.  
It is possible that the group $\GTh_{gen}$ is more tractable and it is often 
denoted by $\GTh_0$ (see, for example, \cite[Section 0.1]{HS-fund-groups}). 
The group $\GTh_{gen}$ obviously contains $\GTh$ as a subgroup. 
 
The idea of approximating elements of $\GTh$ and $\GTh_{gen}$ was originally 
suggested in paper \cite{HS} by D. Harbater and L. Schneps. 
We would also like to mention papers 
\cite{Guillot} and \cite{Guillot1} in which P. Guillot developed and studied 
similar constructions for the group $\GTh_{gen}$. Since P. Guillot used 
a very different definition of $\GTh_{gen}$, it is not easy to compare 
the groupoid $\GTSh$ developed in this paper to the constructions presented 
in \cite{Guillot}, \cite{Guillot1}. 
 
\bigskip
\noindent
{\bf The groupoid $\GTSh$ in a nutshell.} Our starting point is the poset 
$\NFI_{\PB_3}(\B_3)$ of finite index normal subgroups $\N$ of $\B_3$ such that 
$\N \le \PB_3$, i.e.
\begin{equation}
\label{NFI-intro}
\NFI_{\PB_3}(\B_3) := \{\N \unlhd \B_3 ~|~ |\B_3 : \N| < \infty, ~ \N \le \PB_3\}.
\end{equation}
Since $\PB_3$ (and $\B_3$) is residually finite, the poset $\NFI_{\PB_3}(\B_3)$ is infinite. 

For $\N \in \NFI_{\PB_3}(\B_3)$, we denote by $N_{\ord}$ the least 
common multiple of the orders of elements $x_{12} \N$, $x_{23} \N$ and $c \N$ in $\PB_3/\N$. 
Moreover, we set $\N_{\F_2} : = \F_2 \cap \N$, where $\F_2$ is identified with the 
subgroup  $\lan x_{12}, x_{23} \ran$ of $\PB_3$. 

For $\N \in \NFI_{\PB_3}(\B_3)$, we consider pairs $(m, f) \in \ZZ \times \F_2$ that satisfy 
the hexagon relations modulo $\N$:
\begin{equation}
\label{hexa1-N-intro}
\si_1^{2m+1}  \, f^{-1} \si_2^{2m+1} f \, \N ~ = ~ 
f^{-1} \si_1 \si_2 x_{12}^{-m} c^m \, \N,
\end{equation} 
\begin{equation}
\label{hexa11-N-intro}
f^{-1} \si_2^{2m+1} f \, \si_1^{2m+1} \,  \N  ~=~ \si_2 \si_1 x_{23}^{-m} c^m \, f \,  \N.
\end{equation}

Due to Proposition \ref{prop:T-m-f}, for every such pair $(m, f)$, the formulas 
\begin{equation}
\label{Tmf-dfn}
T_{m,f} (\si_1):= \si_1^{2 m+1}\, \N, \qquad
T_{m,f} (\si_2):= f^{-1} \si_2^{2 m+1} f \, \N
\end{equation}
define a group homomorphism $T_{m, f} : \B_3 \to \B_3/\N$. 

A $\GT$-\e{shadow} with the target $\N$ is a pair 
$$
(m + N_{\ord} \ZZ ,f \N_{\F_2}) \in \ZZ/N_{\ord} \ZZ \times [\F_2/\N_{\F_2}, \F_2/\N_{\F_2}]
$$
that satisfies the following conditions: 

\begin{itemize}

\item relations \eqref{hexa1-N-intro},  \eqref{hexa11-N-intro} hold, 

\item $2 m + 1$ represents a unit in $\ZZ/N_{\ord} \ZZ$, and

\item the group homomorphism $T_{m, f} : \B_3 \to \B_3/\N$ is onto.

\end{itemize}
We denote by $\GT(\N)$ the set of $\GT$-shadows with the target $\N$
and by $[m, f]$ the $\GT$-shadow represented by a pair $(m, f) \in \ZZ \times \F_2$.
The set $\GT(\N)$ is finite since it is a subset of a finite set. 

Using equation \eqref{rho-N-Tmf}, it is not hard to see that, for every 
$[m, f]  \in \GT(\N)$, $\ker(T_{m, f})$ belongs to the poset $\NFI_{\PB_3}(\B_3)$.
Moreover, since $T_{m, f}$ is onto, it induces an isomorphism of the quotient groups:
$$
T^{\isom}_{m,f} : \B_3/\K \iso \B_3/\N, 
$$
where $\K : = \ker(T_{m, f})$. 
 
The set $\Ob(\GTSh)$ of objects of the groupoid $\GTSh$ is the poset $\NFI_{\PB_3}(\B_3)$. 
Moreover, for $\K, \N \in \NFI_{\PB_3}(\B_3)$, the set $\GTSh(\K, \N)$ of morphisms 
from $\K$ to $\N$ is the subset of $\GT$-shadows 
$[m, f] \in \GT(\N)$ for which $\K =  \ker(T_{m, f})$. 

The composition of morphisms $[m_1, f_1] \in \GTSh(\N^{(2)}, \N^{(1)})$,  
$[m_2, f_2] \in \GT(\N^{(3)}, \N^{(2)})$ is defined by the formula
$$
[m_1, f_1] \circ [m_2,  f_2]  := [2m_1 m_2 + m_1 +m_2 , f_1 E_{m_1, f_1}(f_2)],
$$
where $E_{m_1, f_1}$ is the endomorphism of $\F_2$ defined by the equations 
$E_{m_1,f_1}(x) := x^{2 m_1 + 1}$, $E_{m_1,f_1}(y) := f_1^{-1} y^{2 m_1 + 1} f_1$
(for more details, see Theorem \ref{thm:GTSh}).

It is important that $\Ob(\GTSh)$ is a poset.
In Subsection \ref{sec:cR-N-H}, we show that, if $\N \le \H$, $\N, \H \in \NFI_{\PB_3}(\B_3)$, 
then we have a natural \e{reduction map}: 
\begin{equation}
\label{cR-N-H-intro}
\cR_{\N, \H} : \GT(\N) \to \GT(\H).  
\end{equation}
In Section \ref{sec:ML}, this map plays an important role in connecting the groupoid 
$\GTSh$ to the group $\GTh_{gen}$. 

Although the groupoid $\GTSh$ is infinite, the connected component 
$\GTSh_{\conn}(\N)$ of an object $\N \in  \NFI_{\PB_3}(\B_3)$ is always a finite 
groupoid. An object $\N$ of the groupoid $\GTSh$ is called \e{isolated} if its connected 
component $\GTSh_{\conn}(\N)$ has exactly one object. In this case, 
$\GTSh(\N, \N) = \GT(\N)$ is a (finite) group and the groupoid 
$\GTSh_{\conn}(\N)$ may be identified with this group. 
In Subsection \ref{sec:isolated}, we show that the subposet  
$\NFI^{isolated}_{\PB_3}(\B_3) \subset  \NFI_{\PB_3}(\B_3)$
of isolated objects of $\GTSh$ is cofinal, i.e., for every $\N \in \NFI_{\PB_3}(\B_3)$, there exists 
$\ti{\N} \in \NFI^{isolated}_{\PB_3}(\B_3)$ such that $\ti{\N} \le \N$. More precisely, 
due to Proposition \ref{prop:N-diamond}, for every $\N \in \NFI_{\PB_3}(\B_3)$, 
the intersection of all objects of the connected component $\GTSh_{\conn}(\N)$
is an isolated object $\N^{\dia}$ of $\GTSh$ such that $\N^{\dia} \le \N$. 


\bigskip
\noindent
{\bf The group $\GTh_{gen}$ versus the groupoid $\GTSh$.} 
In Section \ref{sec:GThat-NFI}, we define a natural action of the group $\GTh_{gen}$ 
on the poset $\NFI_{\PB_3}(\B_3)$. This allows us to introduce the transformation 
groupoid  $\GTh^{gen}_{\NFI}$ and a functor 
$$
\PR : \GTh^{gen}_{\NFI} \to \GTSh.
$$
More precisely, to every element $(\hat{m}, \hat{f}) \in \GTh_{gen}$ 
and every $\N \in \NFI_{\PB_3}(\B_3)$, we assign a $\GT$-shadow 
$[m, f]_{\N}$ with the target $\N$, and the formula
$$
\N^{(\hat{m}, \hat{f})} : = \ker(T_{m, f}) 
$$
defines a right action of $\GTh_{gen}$ on the poset $\NFI_{\PB_3}(\B_3)$. 
We can think of the $\GT$-shadow $[m, f]_{\N}$ as an approximation of 
the element $(\hat{m}, \hat{f})$. For this reason, the functor $\PR$ is called the 
\e{approximation functor}. $\GT$-shadows obtained in this way from elements of 
$\GTh_{gen}$ are called \e{genuine} and all the remaining $\GT$-shadows (if any) 
are called \e{fake}. 

In Section \ref{sec:ML}, we show how the topological group $\GTh_{gen}$
can be reconstructed from the groupoid $\GTSh$.  
We observe that, for every $\N \in \NFI^{isolated}_{\PB_3}(\B_3)$, 
$\GT(\N)$ is a finite group, and the reduction map \eqref{cR-N-H-intro} 
allows us to upgrade the assignment 
$$
\N \mapsto \GT(\N), \qquad \N \in \NFI^{isolated}_{\PB_3}(\B_3)
$$
to a functor from the poset  $\NFI^{isolated}_{\PB_3}(\B_3)$ to the category of 
finite groups. We call it the \e{Main Line functor} and denote it by $\ML$. 

Using the approximation functor $\PR: \GTh^{gen}_{\NFI} \to \GTSh$, it is easy to 
construct a natural group homomorphism 
\begin{equation}
\label{Psi-intro}
\Psi : \GTh_{gen} \to \lim(\ML).  
\end{equation}
The main result of this paper is Theorem \ref{thm:lim-ML} which states that 
$\Psi$ is an isomorphism of groups and a homeomorphism of topological 
spaces. ($\GTh_{gen}$ is considered with the subset topology coming from 
the topological space $\Zhat \times \Fh_2$.)

Thanks to Theorem \ref{thm:lim-ML}, we have the following criterion for 
identifying genuine $\GT$-shadows: a $\GT$-shadow $[m, f] \in \GT(\H)$ is genuine 
if and only if $[m, f]$ belongs to the image of the reduction map
$\cR_{\N, \H} : \GT(\N) \to \GT(\H)$ for every $\N \in \NFI_{\PB_3}(\B_3)$ 
such that $\N \le \H$ (see Corollary \ref{cor:genuine-iff}). Equivalently, 
a $\GT$-shadow $[m, f] \in \GT(\H)$ is fake if and only if there exists 
$\N \in \NFI_{\PB_3}(\B_3)$ such that $\N \le \H$ and $[m, f]$ does not 
belong to the image of the reduction map $\cR_{\N, \H} : \GT(\N) \to \GT(\H)$. 

In recent paper \cite{YD-final-paper}, the authors considered a subposet $\Dih$ 
of $\NFI_{\PB_3}(\B_3)$ related to the family of dihedral groups and they 
called $\Dih$ the \e{dihedral poset}. In \cite{YD-final-paper}, it was proved that every 
element of the dihedral poset is an isolated object of the groupoid $\GTSh$ 
and gave an explicit description of the (finite) group $\GT(\K)$ for every $\K \in \Dih$. 
In \cite{YD-final-paper}, the authors also proved that, for every pair $\N, \H \in \Dih$ with $\N \le \H$, 
the natural map $\GT(\N) \to \GT(\H)$ is onto. This result implies that one cannot find 
an example of a fake\footnote{At the time of writing, the authors of this paper do not know a single 
example of a fake $\GT$-shadow.} $\GT$-shadow using only the dihedral poset $\Dih$.

\bigskip
\noindent
{\bf Organization of the paper.} In Section \ref{sec:GTh-gen}, we introduce the group $\GTh_{gen}$. 
We also recall that $\GTh_{gen}$ comes with natural injective homomorphisms to the 
group of continuous automorphisms of $\Fh_2$ and to the group of continuous 
automorphisms of $\wh{\B}_3$. 

Section \ref{sec:GTSh} is the core of this paper. In this section, we introduce the groupoid 
$\GTSh$ of $\GT$-shadows (for $\GTh_{gen}$), define the reduction map (see \eqref{cR-N-H-intro} 
or \eqref{cR-N-H}), discuss connected components of $\GTSh$ and introduce 
isolated objects of $\GTSh$.  

In Section \ref{sec:GThat-NFI}, we introduce the action of the group $\GTh$ on the poset
$\NFI_{\PB_3}(\B_3)$ and define the approximation functor $\PR$ from the transformation groupoid 
$\GTh^{gen}_{\NFI}$ to $\GTSh$. The $\GT$-shadows that belong to the image of $\PR$
are called genuine. 

In Section \ref{sec:ML}, we introduce the Main Line functor $\ML$ and prove 
that $\lim(\ML)$ is isomorphic to the group $\GTh_{gen}$ (see Theorem \ref{thm:lim-ML}). 
In this section, we also prove a criterion for identifying genuine 
$\GT$-shadows (see Corollary \ref{cor:genuine-iff}) and show that 
the group $\GTh_{gen}$ is isomorphic to the group $\GTh_0$ introduced 
in \cite[Section 0.1]{HS-fund-groups} (see Proposition \ref{prop:H-I-H-II}).

Appendix \ref{app:profinite} is devoted to selected statements about profinite groups.

\subsection{Notational conventions}
For a set $X$ with an equivalence relation and $a \in X$ we will denote by $[a]$
the equivalence class which contains the element $a$. 

The notation $\B_n$ (resp. $\PB_n$) is reserved for the Artin braid group 
on $n$ strands (resp. the pure braid group on $n$ strands). $S_n$ denotes the 
symmetric group on $n$ letters. We denote by $\si_1$ and $\si_2$ the standard 
generators of $\B_3$. Furthermore, we set 
$$
x_{12} := \si_1^2, \qquad x_{23}:=\si_2^2, \qquad  
\D := \si_1 \si_2 \si_1, \qquad 
c : = \D^2\,.
$$
We recall \cite[Section 1.3]{Braids} that the element $c$
belongs to the center $\cZ(\PB_3)$ of $\PB_3$ (and the center $\cZ(\B_3)$ of $\B_3$). 
Moreover, $\cZ(\B_3) = \cZ(\PB_3) = \lan c \ran \cong \ZZ$. 

We observe that 
\begin{equation}
\label{si-D-si}
\si_1 \D = \D \si_2, \qquad \si_2 \D = \D \si_1, 
\qquad 
\si_1^{-1} \D = \D \si_2^{-1}, 
\qquad 
\si_2^{-1} \D = \D \si_1^{-1}.
\end{equation}

Using identities \eqref{si-D-si} and $c = \D^2$, it is easy to see that the adjoint 
action of $\B_3$ on $\PB_3$ is given on generators by the formulas: 
\begin{equation}
\label{conj-by-si1}
\si_1 x_{12} \si_1^{-1}  = \si_1^{-1} x_{12} \si_1 = x_{12}, \qquad 
\si_1 x_{23} \si_1^{-1} = x_{23}^{-1} x_{12}^{-1} c, \qquad 
\si_1^{-1} x_{23} \si_1 = x_{12}^{-1} x_{23}^{-1}  c,
\end{equation} 
\begin{equation}
\label{conj-by-si2}
\si_2 x_{12} \si_2^{-1} = x^{-1}_{12} x^{-1}_{23} c, \qquad 
\si_2^{-1} x_{12} \si_2 = x_{23}^{-1} x_{12}^{-1} c \qquad
\si_2 x_{23} \si_2^{-1} = \si_2^{-1} x_{23} \si_2 =  x_{23}\,.
\end{equation}  
Moreover, 
\begin{equation}
\label{conj-by-D}
\D x_{12} \D^{-1} = x_{23}, \qquad 
\D x_{23} \D^{-1} = x_{12}\,.
\end{equation}

It is known \cite[Section 1.3]{Braids} that $\lan x_{12}, x_{23} \ran $ is isomorphic to the free group $\F_2$ on two 
generators and we tacitly identify $\F_2$ with the subgroup $\lan x_{12}, x_{23} \ran $ of $\PB_3$. 
Furthermore, $\PB_3$ is isomorphic to $\F_2 \times \lan c \ran$ 
\cite[Section 1.3]{Braids}. We often by $x, y, z$ the elements
$x_{12}$, $x_{23}$ and $(x_{12} x_{23})^{-1}$, respectively, i.e.
$$
x:=x_{12}, \qquad y:=x_{23}, \qquad 
z:=y^{-1} x^{-1}\,.
$$

We denote by $\te$ and $\tau$ the automorphisms of $\F_2:= \lan x,y \ran$
defined by the formulas
\begin{equation}
\label{theta}
\te(x):= y, \qquad \te(y):= x,
\end{equation}
\begin{equation}
\label{tau}
\tau(x):= y, \qquad \tau(y):= y^{-1} x^{-1}\,.
\end{equation}
By abuse of notation, we will use the same letters $\te$ and $\tau$ for the corresponding continuous 
automorphisms of $\wh{\F}_2$, respectively. 
(See Corollary \ref{cor:extend-uniquely} in Appendix \ref{app:profinite}).

%

For a group $G$, the notation $[G, G]$ is reserved for the commutator subgroup of $G$. 
For a subgroup $H\le G$, the notation $|G:H|$ is reserved for the index of $H$ in $G$. 
For a normal subgroup $H\unlhd G$ of finite index, we denote by $\NFI_{H}(G)$ the poset 
of finite index normal subgroups $\N$ in $G$ such that $\N \le H$. Moreover, 
$\NFI(G) := \NFI_G(G)$, i.e. $\NFI(G)$ is the poset of normal finite index subgroups of a group $G$.
For a subgroup $H\le G$, $\Core_G(H)$ denotes the normal core of $H$ in $G$, i.e. 
$$
\Core_G(H) : = \bigcap_{g \in G} g H g^{-1}\,.
$$

For $\N \in \NFI(G)$, $\cP_{\N}$ denotes the standard (onto) homomorphism 
\begin{equation}
\label{cP-N}
\cP_{\N} : G \to G/\N.
\end{equation}
Moreover, for $\K \in  \NFI(G)$ such that $\K \le \N$, the notation $\cP_{\K, \N}$ is 
reserved for the standard (onto) homomorphism 
\begin{equation}
\label{cP-K-N}
\cP_{\K, \N} : G/\K \to G/\N.
\end{equation}
Every finite group/set is tacitly considered with the discrete topology.

For a group $G$, $\wh{G}$ denotes the profinite completion of $G$. If $G$ is residually finite, 
then we tacitly identify $G$ with its image in $\wh{G}$. For $\N \in \NFI(G)$, $\hcP_{\N}$ denotes 
the standard continuous group homomorphism
\begin{equation}
\label{hcP-N}
\hcP_{\N} : \wh{G} \to G/\N.
\end{equation}

Let $G$ be a residually finite group. Since every group homomorphism $\vf : G \to \wh{H}$
extends uniquely to a continuous group homomorphism from $\wh{G}$ to $\wh{H}$ 
(see Corollary \ref{cor:extend-uniquely} in Appendix \ref{app:profinite}),
we often use the same symbol for this continuous group 
homomorphism $\wh{G} \to \wh{H}$. 

For a prime $p$, $\ZZ_p$ denotes the ring of $p$-adic integers. 

For a category $\cC$, the notation $\Ob(\cC)$ is reserved for the set of objects of $\cC$.
For $a, b \in \Ob(\cC)$, $\cC(a,b)$ denotes the set of morphisms in $\cC$ from $a$ to $b$. 
Every poset $J$ is tacitly considered as the category with $J$ being the set of its objects; 
if $j_1 \le j_2$, then we have exactly one morphism $j_1 \to j_2$; otherwise, there are no 
morphisms from $j_1$ to $j_2$. A subposet $\ti{J} \subset J$ is called \e{cofinal} if
$\forall ~ j \in J$ $\exists~ \ti{j} \in \ti{J}$ such that $\ti{j} \le j$.  
 
\bigskip
\noindent
{\bf Notational quirks.} Paper \cite{GTshadows} develops the groupoid of $\GT$-shadows for 
the original (profinite) version $\GTh$ of the Grothendieck-Teichmueller group \cite[Section 4]{Drinfeld}. 
In consideration of paper \cite{GTshadows}, we should have denoted the groupoid of $\GT$-shadows for $\GTh_{gen}$
by $\GTSh_{gen}$. However, we decided to omit the subscript ``$gen$'' to simplify the notation. This should 
not lead to a confusion because the main focus of this paper is the group $\GTh_{gen}$ and 
the corresponding groupoid of $\GT$-shadows. We should also mention that, paper \cite{GTshadows} 
considers $\GT$-shadows $[m, f]$ that may not satisfy the condition 
\begin{equation}
\label{fN-in-comm-subgroup}
f \N_{\F_2} \in [ \F_2/ \N_{\F_2},  \F_2/ \N_{\F_2} ],
\end{equation}
and $\GT$-shadows $[m, f]$ satisfying \eqref{fN-in-comm-subgroup} are called charming. 
(In fact, in paper \cite{GTshadows}, the authors consider the groupoid $\GTSh$ of $\GT$-shadows 
for $\GTh$ and the subgroupoid $\GTSh^{\hs} \subset \GTSh$ of charming $\GT$-shadows (for $\GTh$).)
In this paper, we impose condition \eqref{fN-in-comm-subgroup} at an earlier stage. Hence we have 
only one groupoid of $\GT$-shadows for $\GTh_{gen}$.

\bigskip
\noindent
{\bf Acknowledgement.} We are thankful to No\'emie Combe, Jaclyn Lang, Winnie Li, 
Adrian Ocneanu, Jessica Radford and Jingfeng Xia for useful discussions. 
We are thankful to Benjamin Enriquez for suggesting us the
idea to use the transformation groupoid of the action of $\GTh_{gen}$ on 
$\NFI_{\PB_3}(\B_3)$. We are thankful to Florian Pop for showing us paper \cite{HerfortRibes}.
We are thankful to Ivan Bortnovskyi, Borys Holikov, Vadym Pashkovskyi and Ihor Pylaiev  
for their amazing patience with details of this paper. V.A.D. acknowledges a partial 
support from the project ``Arithmetic and Homotopic Galois Theory''. This project is, in turn,  
supported by the CNRS France-Japan AHGT International Research Network between 
the RIMS Kyoto University, the LPP of Lille University, and the DMA of ENS PSL. 
V.A.D. also acknowledges Temple University 
for 2021 Summer Research Award. J.J.G. acknowledges the Temple University College of 
Science and Technology Science Scholars Program for their funding in Summer 2020.

\section{The gentle version $\GTh_{gen}$ of the Grothendieck-Teichmueller group}
\label{sec:GTh-gen}

\subsection{The monoid $\big( \Zhat \times \wh{\F}_2, \bullet \big)$}
\label{sec:Zhat-F2hat}

To introduce $\GTh_{gen}$, we denote by $E_{\hat{m}, \hat{f}}$ the following 
group homomorphism from $\F_2$ to $\wh{\F}_2$
\begin{equation}
\label{F2-to-F2-hat}
E_{\hat{m}, \hat{f}}(x) := x^{2\hat{m}+1}, \qquad  
E_{\hat{m}, \hat{f}}(y) :=  \hat{f}^{-1} y^{2\hat{m}+1} \hat{f},
\end{equation}
where $(\hat{m}, \hat{f}) \in \Zhat \times \wh{\F}_2$.

Due to Corollary \ref{cor:extend-uniquely} from Appendix 
\ref{app:profinite}, $E_{\hat{m}, \hat{f}}$ extends uniquely to a continuous endomorphism 
of $\wh{\F}_2$:
\begin{equation}
\label{Emf-hat}
E_{\hat{m}, \hat{f}} :  \wh{\F}_2 \to  \wh{\F}_2\,.
\end{equation}
By abuse of notation, we use the same symbol $E_{\hat{m}, \hat{f}}$ for 
the extension of the homomorphism defined in \eqref{F2-to-F2-hat}.

Let $(\hat{m}_1, \hat{f}_1), (\hat{m}_2, \hat{f}_2) \in \Zhat \times \wh{\F}_2$ and
$$
\hat{m} := 2\hat{m}_1 \hat{m}_2 + \hat{m}_1 + \hat{m}_2, 
\qquad 
\hat{f} :=  \hat{f}_1 E_{\hat{m}_1, \hat{f}_1}(\hat{f}_2).
$$
A direct computation shows that 
$$
E_{\hat{m}_1, \hat{f}_1} \circ E_{\hat{m}_2, \hat{f}_2}(x) = E_{\hat{m}, \hat{f}} (x), 
\qquad
E_{\hat{m}_1, \hat{f}_1} \circ E_{\hat{m}_2, \hat{f}_2}(y) = E_{\hat{m}, \hat{f}}(y).
$$
Hence, applying Corollary \ref{cor:extend-uniquely}, we conclude that 
\begin{equation}
\label{E-homomorphism}
E_{\hat{m}_1, \hat{f}_1} \circ E_{\hat{m}_2, \hat{f}_2} = E_{\hat{m}, \hat{f}}.
\end{equation}

This motivates us to define the following binary operation 
$\bullet$ on $\Zhat \times \wh{\F}_2$
\begin{equation}
\label{bullet}
(\hat{m}_1, \hat{f}_1) \bullet (\hat{m}_2, \hat{f}_2) : =
\big(\, 2\hat{m}_1 \hat{m}_2 + \hat{m}_1 + \hat{m}_2, \, 
\hat{f}_1 E_{\hat{m}_1, \hat{f}_1}(\hat{f}_2) \, \big). 
\end{equation}

Let us prove that 
\begin{prop}  
\label{prop:Z-F2-monoid}
The set $\Zhat \times \wh{\F}_2$ is a monoid with respect to the binary 
operation $\bullet$ (see \eqref{bullet}) and the pair $(0, 1_{ \wh{\F}_2})$ is the 
identity element of this monoid. Moreover, the assignment 
$$
(\hat{m}, \hat{f}) \mapsto E_{\hat{m}, \hat{f}}
$$
defines a homomorphism of monoids $\Zhat \times \wh{\F}_2 \to \End(\wh{\F}_2)$, 
where $\End(\wh{\F}_2)$ is the monoid of continuous endomorphisms 
of $\wh{\F}_2$.
\end{prop}  
\begin{proof}
It is easy to see that $(0, 1_{ \wh{\F}_2})$ is the identity element of the 
magma $\big( \Zhat \times \wh{\F}_2, \bullet \big)$. So let us prove the 
associativity of $\bullet$. 

For $(\hat{m}_1, \hat{f}_1), (\hat{m}_2, \hat{f}_2), (\hat{m}_3, \hat{f}_3) \in \Zhat \times \wh{\F}_2 $, we have
\begin{equation}
\label{assoc-LHS}
\big( (\hat{m}_1, \hat{f}_1) \bullet (\hat{m}_2, \hat{f}_2) \big)  \bullet (\hat{m}_3, \hat{f}_3) = 
\big(2 \hat{q} \hat{m}_3 + \hat{q} + \hat{m}_3, \hat{g} E_{\hat{q},\hat{g}} (\hat{f}_3) \big)
\end{equation}
and 
\begin{equation}
\label{assoc-RHS}
(\hat{m}_1, \hat{f}_1) \bullet \big( (\hat{m}_2, \hat{f}_2) \bullet (\hat{m}_3, \hat{f}_3)  \big) =
(2 \hat{m}_1 \hat{k} + \hat{m}_1 + \hat{k}, \hat{f}_1 E_{\hat{m}_1,\hat{f}_1} (\hat{h}) ),
\end{equation}
where $(\hat{q}, \hat{g}) :=(\hat{m}_1, \hat{f}_1) \bullet (\hat{m}_2, \hat{f}_2)$ and 
$(\hat{k}, \hat{h}) := (\hat{m}_2, \hat{f}_2) \bullet (\hat{m}_3, \hat{f}_3)$. 

Using $\hat{q} := 2 \hat{m}_1 \hat{m}_2 + \hat{m}_1 + \hat{m}_2$ and 
$\hat{k} := 2 \hat{m}_2 \hat{m}_3 + \hat{m}_2 + \hat{m}_3$, it is easy to see that 
$$
2 \hat{q} \hat{m}_3 + \hat{q} + \hat{m}_3 = 2 \hat{m}_1 \hat{k} + \hat{m}_1 + \hat{k}.
$$

Using \eqref{E-homomorphism} and the fact that $E_{\hat{m_1},\hat{f}_1}$ is an endomorphism 
of $\wh{\F}_2$, we can rewrite $\hat{g} E_{\hat{q},\hat{g}} (\hat{f}_3)$ as follows 
$$
\hat{g} E_{\hat{q},\hat{g}} (\hat{f}_3) = 
\hat{f_1} E_{\hat{m_1},\hat{f}_1}(\hat{f}_2)\, 
E_{\hat{m_1},\hat{f}_1} \circ E_{\hat{m_2},\hat{f}_2} (\hat{f}_3) =
\hat{f_1} E_{\hat{m_1},\hat{f}_1}\big( \hat{f}_2 E_{\hat{m_2},\hat{f}_2} (\hat{f}_3) \big).
$$
Thus $\hat{g} E_{\hat{q},\hat{g}} (\hat{f}_3) = \hat{f}_1 E_{\hat{m}_1,\hat{f}_1} (\hat{h})$ and 
the associativity of $\bullet$ is proved. 

Since $E_{0, 1_{\wh{\F}_2}} = \id_{\wh{\F}_2}$, the last statement of the proposition follows from 
\eqref{E-homomorphism}.  
\end{proof}

\begin{remark}  
\label{rem:cyclotomic}
It is easy to see that, if $(\hat{m}, \hat{f}) = (\hat{m}_1, \hat{f}_1) \bullet (\hat{m}_2, \hat{f}_2) $, then 
\begin{equation}
\label{cyclotomic-hat}
2\hat{m}+1 = (2\hat{m}_1+1) (2\hat{m}_2+1).  
\end{equation}
\end{remark}

\subsection{The monoid $\GTh_{gen, mon}$ and the group $\GTh_{gen}$} 
\label{sec:GTh-gen-mon-GTh-gen}

Let us denote by $\GTh_{gen, mon}$ the subset of $\Zhat \times \wh{\F}_2$ that consists of pairs
\begin{equation}
\label{the-pairs-hat}
(\hat{m}, \hat{f}) \in \Zhat \times [\wh{\F}_2, \wh{\F}_2]^{top.\,cl.}
\end{equation}
satisfying the hexagon relations 
\begin{equation}
\label{hexa1-hat}
\si_1^{2 \hat{m} + 1}   \hat{f}^{-1} \si_2^{2 \hat{m} + 1}  \hat{f} =   
\hat{f}^{-1} \si_1 \si_2\, x_{12}^{-\hat{m}} c^{\hat{m}}\,,
\end{equation}
\begin{equation}
\label{hexa11-hat}
\hat{f}^{-1} \si_2^{2\hat{m}+1} \hat{f} \, \si_1^{2\hat{m}+1}  = \si_2 \si_1 x_{23}^{-\hat{m}} c^{\hat{m}} \, \hat{f} \,.
\end{equation}

Let us prove that\footnote{See \cite[Lemma 1]{Leila-survey}.}, 
\begin{prop}  
\label{prop:endom-B3-hat}
For every $(\hat{m}, \hat{f}) \in \GTh_{gen, mon}$, the formulas
\begin{equation}
\label{Tmf-hat}
T_{\hat{m}, \hat{f}} (\si_1):= \si_1^{2\hat{m}+1}\,, \qquad
T_{\hat{m}, \hat{f}} (\si_2):= \hat{f}^{-1} \si_2^{2\hat{m}+1} \hat{f}
\end{equation}
define a group homomorphism $T_{\hat{m}, \hat{f}}  : \B_3 \to \wh{\B}_3$ such that 
\begin{equation}
\label{Tmf-c}
T_{\hat{m}, \hat{f}} (c) = c^{2\hat{m}+1}\,.
\end{equation}
The homomorphism $T_{\hat{m}, \hat{f}}$ extends uniquely to a continuous 
endomorphism of $\wh{\B}_3$ and
\begin{equation}
\label{Tmf-hat-F2}
T_{\hat{m}, \hat{f}} \big|_{\wh{\F}_2} = E_{\hat{m}, \hat{f}}.
\end{equation}
\end{prop}  
\begin{proof}
We need to verify that 
\begin{equation}
\label{relation-ok}
T_{\hat{m}, \hat{f}} (\si_1) T_{\hat{m}, \hat{f}} (\si_2) T_{\hat{m}, \hat{f}} (\si_1) 
\overset{?}{=}
T_{\hat{m}, \hat{f}} (\si_2) T_{\hat{m}, \hat{f}} (\si_1) T_{\hat{m}, \hat{f}} (\si_2) 
\end{equation}
or equivalently 
\begin{equation}
\label{relation-ok1}
\si_1^{2\hat{m}+1}\, \hat{f}^{-1} \si_2^{2\hat{m}+1} \hat{f} \, \si_1^{2\hat{m}+1} 
\overset{?}{=}
 \hat{f}^{-1} \si_2^{2\hat{m}+1} \hat{f} \, 
 \si_1^{2\hat{m}+1} \,  \hat{f}^{-1} \si_2^{2\hat{m}+1} \hat{f} \,.
\end{equation}

Applying \eqref{hexa1-hat} to the left hand side of \eqref{relation-ok1}, we get 
\begin{equation}
\label{LHS-ok}
\si_1^{2\hat{m}+1}\, \hat{f}^{-1} \si_2^{2\hat{m}+1} \hat{f} \, \si_1^{2\hat{m}+1}  = 
\hat{f}^{-1} \si_1 \si_2\, x_{12}^{-\hat{m}} c^{\hat{m}}
 \, \si_1^{2\hat{m}+1} = \hat{f}^{-1} \D c^{\hat{m}}\,.
\end{equation}

To take care of the right hand side of \eqref{relation-ok1} we notice that, 
for every $\hat{t} \in \Zhat$,  
\begin{equation}
\label{Delta-hat-hat-Delta}
\D \si_1^{\hat{t}} = \si_2^{\hat{t}} \D.
\end{equation}

Indeed, for every $\N \in \NFI(\B_3)$, there exists $t_{\N} \in \ZZ$ such that 
$\hcP_{\N}(\D \si_1^{\hat{t}}) = \D \si_1^{t_{\N}} \N$ and 
$\hcP_{\N}(\si_2^{\hat{t}} \D) = \si_2^{t_{\N}} \D \N$. Since 
$\D \si_1^{k} =  \si_2^{k} \D$ for every integer $k$, 
relation \eqref{Delta-hat-hat-Delta} holds. 

Applying \eqref{hexa1-hat} to the right hand side of \eqref{relation-ok1} 
and using \eqref{Delta-hat-hat-Delta}, we get 
$$
 \hat{f}^{-1} \si_2^{2\hat{m}+1} \hat{f} \, \si_1^{2\hat{m}+1} \,  \hat{f}^{-1} \si_2^{2\hat{m}+1} \hat{f} = 
 \hat{f}^{-1} \si_2^{2\hat{m}+1} \hat{f} 
 \hat{f}^{-1} \si_1 \si_2\, x_{12}^{-\hat{m}} c^{\hat{m}} = 
$$
$$
\hat{f}^{-1} \si_2^{2\hat{m}}\, \D  x_{12}^{-\hat{m}} c^{\hat{m}} = 
\hat{f}^{-1} \si_2^{2\hat{m}}\, x_{23}^{-\hat{m}} \D  c^{\hat{m}} = 
\hat{f}^{-1} \D c^{\hat{m}}\,.
$$

Combining this result with \eqref{LHS-ok}, we see that relation 
\eqref{relation-ok1} indeed holds. 

Due to \eqref{LHS-ok}, we have 
\begin{equation}
\label{Tmf-Delta}
T_{\hat{m}, \hat{f}} (\D) = \hat{f}^{-1} \D c^{\hat{m}}\,.
\end{equation}

Applying \eqref{hexa11-hat} to $T_{\hat{m}, \hat{f}} (\D) 
= \si_1^{2\hat{m}+1}\, \hat{f}^{-1} \si_2^{2\hat{m}+1} \hat{f} \, \si_1^{2\hat{m}+1} $
and using \eqref{Delta-hat-hat-Delta}, we see that 
\begin{equation}
\label{Tmf-Delta1}
T_{\hat{m}, \hat{f}} (\D) = \si_1^{2\hat{m}+1}\,  \si_2 \si_1 x_{23}^{-\hat{m}} c^{\hat{m}} \, \hat{f} =
 \si_1^{2\hat{m}}  \D x_{23}^{-\hat{m}} c^{\hat{m}} \, \hat{f} = \D c^{\hat{m}} \, \hat{f}\,.
\end{equation}

Combining \eqref{Tmf-Delta} with \eqref{Tmf-Delta1}, we get 
$$
T_{\hat{m}, \hat{f}} (c) = T_{\hat{m}, \hat{f}} (\D) T_{\hat{m}, \hat{f}} (\D)   = 
 \D c^{\hat{m}} \, \hat{f}  \hat{f}^{-1} \D c^{\hat{m}} = c^{2 \hat{m} + 1}.
$$ 

Thus \eqref{Tmf-c} is proved. 

The third statement of the proposition follows from Corollary \ref{cor:extend-uniquely} and 
the proof of the last statement is straightforward. 
\end{proof}

\bigskip

\begin{prop}  
\label{prop:GT-gen-mon-submonoid}
The subset $\GTh_{gen, mon}$ of $\Zhat \times \wh{\F}_2$ is 
a submonoid of $\big(\Zhat \times \wh{\F}_2, \bullet \big)$. 
The assignment 
\begin{equation}
\label{GTh-mon-to-End-F2hat}
(\hat{m}, \hat{f}) \mapsto E_{\hat{m}, \hat{f}}
\end{equation}
defines an injective homomorphism of monoids from $\GTh_{gen, mon}$ to 
the monoid of continuous endomorphisms of $\wh{\F}_2$.
Similarly, the assignment 
\begin{equation}
\label{GTh-mon-to-End-B3hat}
(\hat{m}, \hat{f}) \mapsto T_{\hat{m}, \hat{f}}
\end{equation}
defines an injective homomorphism of monoids from $\GTh_{gen, mon}$ to 
the monoid of continuous endomorphisms of $\wh{\B}_3$. 
\end{prop}
\begin{proof}
Let $(\hat{m}_1, \hat{f}_1), (\hat{m}_2, \hat{f}_2) \in \GTh_{gen, mon}$ and
$(\hat{m}, \hat{f}) := (\hat{m}_1, \hat{f}_1) \bullet (\hat{m}_2, \hat{f}_2)$. 

Since $E_{\hat{m}_1, \hat{f}_1}$ is a continuous group homomorphism and 
$\hat{f}_2 \in [\wh{\F}_2, \wh{\F}_2]^{top.\,cl.}$, 
$E_{\hat{m}_1, \hat{f}_1}(\hat{f}_2)$ also belongs to $[\wh{\F}_2, \wh{\F}_2]^{top.\,cl.}$.
Hence 
$$
\hat{f} := \hat{f}_1 E_{\hat{m}_1, \hat{f}_1}(\hat{f}_2) ~\in~ [\wh{\F}_2, \wh{\F}_2]^{top.\,cl.}\,.
$$

\bigskip

Let us prove that the pair $(\hat{m}, \hat{f})$ satisfies hexagon relations 
\eqref{hexa1-hat} and  \eqref{hexa11-hat}.

Applying $T_{\hat{m}_1, \hat{f}_1}$ to the first hexagon relation for $(\hat{m}_2, \hat{f}_2)$ 
and using identities \eqref{Tmf-c}, \eqref{Tmf-hat-F2} we get 
\begin{equation}
\label{hexa1-mf-hat}
\begin{array}{c}
\si_1^{(2 \hat{m_2} + 1)(2 \hat{m_1} + 1)}  E_{\hat{m}_1, \hat{f}_1}(\hat{f}_2)^{-1} \hat{f}_1^{-1}  
\si_2^{(2\hat{m}_2 + 1)(2 \hat{m_1} + 1)}  \hat{f}_1 E_{\hat{m}_1, \hat{f}_1}(\hat{f}_2) \, = \\[0.28cm]   
E_{\hat{m}_1, \hat{f}_1}(\hat{f}_2)^{-1}  \si_1^{2\hat{m}_1+1}  
\hat{f}_1^{-1} \si_2^{2\hat{m}_1+1}  \hat{f}_1\, 
x_{12}^{-\hat{m}_2(2 \hat{m}_1+1)} c^{\hat{m}_2(2 \hat{m}_1+1)}\,.
\end{array}
\end{equation}

Using \eqref{cyclotomic-hat}, the first hexagon relation for $(\hat{m}_1, \hat{f}_1)$ and \eqref{hexa1-mf-hat}, we get
$$
\si_1^{2 \hat{m} + 1} \hat{f}^{-1} 
\si_2^{2\hat{m} + 1}  \hat{f} \, = \,
\hat{f}^{-1} \si_1 \si_2
x_{12}^{-\hat{m}} c^{\hat{m}}\,.
$$
 
Thus the pair $(\hat{m}, \hat{f})$ satisfies \eqref{hexa1-hat}.  

Similarly, applying $T_{\hat{m}_1, \hat{f}_1}$ to the second hexagon relation for $(\hat{m}_2, \hat{f}_2)$,  
and using identities \eqref{Tmf-c}, \eqref{Tmf-hat-F2}, the second hexagon relation for $(\hat{m}_1, \hat{f}_1)$ and  
\eqref{cyclotomic-hat}, one can show that the pair $(\hat{m}, \hat{f})$ also satisfies \eqref{hexa11-hat}. 

We proved that the subset $\GTh_{gen, mon}$ is closed with respect to the binary operation $\bullet$. 

It is easy to see that the pair $(0, 1_{\wh{\F}_2})$ satisfies hexagon relations \eqref{hexa1-hat}
and  \eqref{hexa11-hat}.  Thus the first statement of the proposition is proved. 

Due to the second statement of Proposition \ref{prop:Z-F2-monoid}, the assignment in 
\eqref{GTh-mon-to-End-F2hat} is a homomorphism of monoids. To prove that this homomorphism 
is injective\footnote{A similar statement was mentioned in \cite[Section 0.1]{HS-fund-groups} without a proof.}, 
we will use Theorem B  from paper \cite{HerfortRibes} by W. Herfort and L. Ribes.   

If $E_{\hat{m}_1, \hat{f}_1} = E_{\hat{m}_2, \hat{f}_2}$, then 
\begin{equation}
\label{E-equals-E}
x^{2\hat{m}_1 + 1} = x^{2\hat{m}_2 + 1}, \qquad 
\hat{f}_1^{-1} y^{2\hat{m}_1 + 1} \hat{f}_1 = \hat{f}_2^{-1} y^{2\hat{m}_2 + 1} \hat{f}_2\,.
\end{equation}

The first equation in \eqref{E-equals-E} implies that $x^{2(\hat{m}_2- \hat{m}_1)} = 1$ and 
hence $2(\hat{m}_2- \hat{m}_1) = 0$.  Since $\ZZ_p$ is an integral domain for every prime $p$
and $\displaystyle \wh{\ZZ} ~\cong \prod_{p \txt{ is prime}} \ZZ_p$, we conclude that 
$\hat{m}_1 = \hat{m}_2$.   

We set $\hat{m} : = \hat{m}_1= \hat{m}_2$ and $\hat{w} = \hat{f}_1 \hat{f}_2^{-1}$ 
The second equation in \eqref{E-equals-E} implies that $\hat{w}$ belongs to the centralizer 
of $y^{2\hat{m} + 1}$. 

We consider the subgroup $\{y^{\hat{n}} : \hat{n} \in \Zhat \} \le \wh{\F}_2$ and notice that, 
for every $\hat{m} \in \Zhat$, $y^{2\hat{m} + 1}$ is a non-trivial element of $\{y^{\hat{n}} : \hat{n} \in \Zhat \}$. 
Indeed, the component of $2\hat{m} + 1$ in $\ZZ_2$ is a unit in $\ZZ_2$. Therefore, $2\hat{m}+1$
cannot be zero in $\Zhat$ and hence $y^{2\hat{m} + 1} \neq 1$.

Applying \cite[Theorem B]{HerfortRibes} to $\hat{w} \in C_{\wh{\F}_2}( y^{2\hat{m} + 1})$, we conclude that 
$\hat{w} \in \{y^{\hat{n}} : \hat{n} \in \Zhat \}$. 

Since $\hat{f}_1, \hat{f}_2 \in  [\wh{\F}_2, \wh{\F}_2]^{top.\,cl.}$ and the intersection 
$\{y^{\hat{n}} : \hat{n} \in \Zhat \} \cap  [\wh{\F}_2, \wh{\F}_2]^{top.\,cl.}$ is 
trivial\footnote{To prove that the subgroup 
$\{y^{\hat{n}} : \hat{n} \in \Zhat \} \cap  [\wh{\F}_2, \wh{\F}_2]^{top.\,cl.}$ is trivial, consider 
homomorphisms $\psi$ from $\F_2$ to finite groups such that $\psi(x) = 1$.}, we conclude that 
$\hat{w} = 1$ and hence $\hat{f}_2 = \hat{f}_1$.  

We proved that the homomorphism of monoids $\GTh_{gen, mon} \to \End(\wh{\F}_2)$ is injective. 

\bigskip

To prove that the assignment in \eqref{GTh-mon-to-End-B3hat} is a homomorphism 
of monoids, we need to show that, 
\begin{equation}
\label{Tmf-hat-identity}
T_{0, 1_{\wh{\F}_2} }= \id_{\wh{\B}_3}
\end{equation}
and, for all $(\hat{m}_1, \hat{f}_1), \, (\hat{m}_2, \hat{f}_2) \in \GTh_{gen, mon}$, we have
\begin{equation}
\label{Tmf-hat-functor}
T_{\hat{m}_1, \hat{f}_1} \circ T_{\hat{m}_2, \hat{f}_2} = T_{\hat{m}, \hat{f}},
\end{equation}
where $(\hat{m}, \hat{f})= (\hat{m}_1, \hat{f}_1) \bullet (\hat{m}_2, \hat{f}_2)$.

Applying $T_{\hat{m}_1, \hat{f}_1} \circ T_{\hat{m}_2, \hat{f}_2}$ and $T_{\hat{m}, \hat{f}}$ 
to the generators $\si_1, \si_2$ of $\B_3$, we see that 
$$
T_{\hat{m}_1, \hat{f}_1} \circ T_{\hat{m}_2, \hat{f}_2} \big|_{\B_3} = 
T_{\hat{m}, \hat{f}} \big|_{\B_3}.
$$

Since the maps  $T_{\hat{m}_1, \hat{f}_1} \circ T_{\hat{m}_2, \hat{f}_2}$ and 
$T_{\hat{m}, \hat{f}}$ are continuous, they agree on a dense subset $\B_3$ of $\wh{\B}_3$ and 
$\wh{\B}_3$ is Hausdorff, equation \eqref{Tmf-hat-functor} holds. 

The same argument works for \eqref{Tmf-hat-identity}.

The injectivity of the homomorphism $\GTh_{gen, mon} \to \End(\wh{\B}_3)$ follows 
from the injectivity of the homomorphism $\GTh_{gen, mon} \to \End(\wh{\F}_2)$ and 
identity \eqref{Tmf-hat-F2}. 
\end{proof}

\begin{defi}  
\label{dfn:GTh-gen}
$\GTh_{gen}$ is the group of invertible elements of the monoid $\GTh_{gen, mon}$.
\end{defi}  

\begin{remark}  
\label{rem:GTh-gen-PaB}
As far as we know, the group $\GTh_{gen}$ was introduced in \cite{HS-fund-groups}
and, in \cite{HS-fund-groups}, it is denoted by $\GTh_0$. More precisely, $\GTh_0$ consists 
of elements $(\hat{m}, \hat{f}) \in \Zhat \times [\wh{\F}_2, \wh{\F}_2]^{top.\,cl.}$ satisfying 
\begin{equation}
\label{H-I-hat}
\hat{f} \te(\hat{f}) = 1_{\wh{\F}_2}\,,
\end{equation}
\begin{equation}
\label{H-II-hat}
\tau^2(y^{\hat{m}} \hat{f})  \tau(y^{\hat{m}} \hat{f})  y^{\hat{m}} \hat{f}   = 1_{\wh{\F}_2}\,,
\end{equation}
and the appropriate invertibility condition. Please see Section \ref{sec:simple-hexagons}, in which 
we prove that $\GTh_{gen}$ indeed coincides with $\GTh_0$ 
introduced in \cite[Introduction]{HS-fund-groups}.
\end{remark}
\begin{remark}  
\label{rem:topology}
Since $\Zhat \times \wh{\F}_2$ is naturally a topological space and 
$\GTh_{gen}$ is a subset of  $\Zhat \times \wh{\F}_2$, the set $\GTh_{gen}$
is equipped with the subset topology. It is not obvious that $\GTh_{gen}$ is a topological 
group with respect to this topology. This statement follows easily from Theorem 
\ref{thm:lim-ML} proved in Section \ref{sec:ML}. 
\end{remark}
\begin{remark}  
\label{rem:PaB}
Using \cite[Theorem 6.2.4]{Fresse1} (see also \cite[Appendix A.3]{GTshadows}),
one can show that $\GTh_{gen}$ is a subgroup of the group 
$\GTh_{\le 3}$ of continuous automorphisms of the truncation 
$\wh{\PaB}^{\le 3}$ of the operad $\wh{\PaB}$.
\end{remark}
\begin{remark}  
\label{rem:virtual-cyclotomic}
It is easy to see that, for every $(\hat{m}, \hat{f}) \in \GTh_{gen}$, 
the endomorphism $E_{\hat{m}, \hat{f}}$ (resp. $T_{\hat{m}, \hat{f}}$) of $\wh{\F}_2$ 
(resp. $\wh{\B}_3$) is invertible. Moreover, due to Proposition \ref{prop:GT-gen-mon-submonoid}, 
the assignments 
$$
(\hat{m}, \hat{f})  \to E_{\hat{m}, \hat{f}}, \qquad 
(\hat{m}, \hat{f})  \to T_{\hat{m}, \hat{f}} 
$$
are injective group homomorphisms from $\GTh_{gen}$ to the group 
of continuous automorphisms of $\wh{\F}_2$ and $\wh{\B}_3$, respectively. 
Due to Remark \ref{rem:cyclotomic}, the formula 
\begin{equation}
\label{virtual-cyclotomic}
\chi_{vir}(\hat{m}, \hat{f}) := 2\hat{m}+1
\end{equation}
defines a group homomorphism $\chi_{vir} : \GTh_{gen} \to  \Zhat^{\times}$, 
where $\Zhat^{\times}$ is the group of units of the ring $\Zhat$. 
We call $\chi_{vir}$ the \e{virtual cyclotomic character}.  Using the Ihara embedding 
$\Ih : G_{\QQ} \into \GTh$ (see \cite[Section 1]{Ihara}) and the surjectivity of the 
cyclotomic character $\chi :  G_{\QQ} \to \Zhat^{\times}$, one can show that 
the group homomorphism $\chi_{vir} : \GTh_{gen} \to  \Zhat^{\times}$ is surjective.  
\end{remark}
\begin{remark}  
\label{rem:NikolovSegal}
Let $G$ be a profinite group with a dense finitely generated subgroup 
(e.g. $G = \wh{\F}_2$). Due to \cite[Theorem 1.1]{NikolovSegal}, every 
endomorphism of $G$ is continuous. Moreover, due to 
\cite[Theorem 1.3]{NikolovSegal}, $[G, G]$ is a closed subgroup 
of $G$. In particular, $[\wh{\F}_2, \wh{\F}_2]^{top.\,cl.} =  [\wh{\F}_2, \wh{\F}_2]$. 
However, in this paper, we do not use Theorems 1.1 and 1.3 from \cite{NikolovSegal}. 
\end{remark}

\section{The groupoid $\GTSh$}
\label{sec:GTSh}

For every $\N \in \NFI_{\PB_3}(\B_3)$, we set 
\begin{equation}
\label{N-ord}
N_{\ord} := \lcm( \ord( x_{12} \N ), \ord( x_{23} \N), \ord( c \N ) )
\end{equation}
and
\begin{equation}
\label{N-F-2}
\N_{\F_2} := \N \cap \F_2\,.
\end{equation}
It is clear that $\N_{\F_2} \in \NFI(\F_2)$.  

We say that a pair $(m,f) \in \ZZ \times \F_2 $ satisfies the hexagon relations 
modulo $\N$ if 
\begin{equation}
\label{hexa1}
\si_1^{2m+1}  \, f^{-1} \si_2^{2m+1} f \, \N ~ = ~ 
f^{-1} \si_1 \si_2 x_{12}^{-m} c^m \, \N,
\end{equation} 
\begin{equation}
\label{hexa11}
f^{-1} \si_2^{2m+1} f \, \si_1^{2m+1} \,  \N  ~=~ \si_2 \si_1 x_{23}^{-m} c^m \, f \,  \N.
\end{equation}
Since $N_{\ord}$ is the least common multiple of the orders of 
the elements $x_{12} \N$,  $x_{23} \N$, $c \N$ and $\N_{\F_2} \le \N$,
we see that, if a pair $(m,f) \in \ZZ \times \F_2$ satisfies \eqref{hexa1} and \eqref{hexa11}, 
then so does the pair $(m + t N_{\ord}, f h)$ for any $t \in \ZZ$ and any $h \in \N_{\F_2}$. 

\begin{defi}  
\label{dfn:GT-pairs}
A $\GT$-\e{pair with the target} $\N$ is a pair 
\begin{equation}
\label{GT-pair}
(m + N_{\ord} \ZZ ,f \N_{\F_2}) \in \ZZ/N_{\ord} \ZZ \times \F_2/\N_{\F_2}
\end{equation}
satisfying relations \eqref{hexa1} and \eqref{hexa11}. 
A $\GT$-pair \eqref{GT-pair} is called \e{charming} if 
\begin{itemize}

\item $2m+1$ represents a unit in the ring $\ZZ/ N_{\ord} \ZZ$ and 

\item $f \N_{\F_2} \in [\F_2/\N_{\F_2}, \F_2/\N_{\F_2}]$, or equivalently  
the coset $f \N_{\F_2}$ can be represented by an element in 
the commutator subgroup $[\F_2, \F_2]$ of $\F_2$. 

\end{itemize}
\end{defi}  
We denote by $\GT_{pr}(\N)$ (resp. $\GT^{\hs}_{pr}(\N)$) the set of $\GT$-pairs
(resp. the set of charming $\GT$-pairs) with the target $\N$. From now on, we denote 
by $[m, f]$ the $\GT$-pair represented by $(m, f) \in \ZZ \times \F_2$. 

\bigskip

The importance of the hexagon relations is emphasized by the following proposition: 
\begin{prop}  
\label{prop:T-m-f}
For every $[m,f] \in \GT_{pr}(\N)$, the formulas 
$$
T_{m,f} (\si_1):= \si_1^{2 m+1} \N, \qquad
T_{m,f} (\si_2):= f^{-1} \si_2^{2 m+1} f \N
$$
define a group homomorphism $T_{m,f}: \B_3 \to \B_3/ \N$.
\end{prop}  
\begin{proof}
Since $\B_3 = \lan  \si_1, \si_2 \, | \,  \si_1\si_2 \si_1 =  \si_2 \si_1 \si_2 \ran$, 
it suffices to verify that 
\begin{equation}
\label{braid-relation}
T_{m,f} (\si_1) T_{m,f} (\si_2) T_{m,f} (\si_1)
\overset{?}{=} T_{m,f} (\si_2) T_{m,f} (\si_1) T_{m,f} (\si_2).
\end{equation}

Using \eqref{hexa1}, we rewrite the left hand side of \eqref{braid-relation} as 
\begin{equation}
\label{LHS-hexa1}
(\si_1^{2m+1} f^{-1} \si_2^{2m+1} f)  \si_1^{2m+1}\, \N = 
 f^{-1} \si_1 \si_2 x_{12}^{-m} c^m \,  \si_1^{2m+1}\, \N =  
 f^{-1} \D c^m \, \N,
\end{equation}
where $\D: = \si_1 \si_2 \si_1$. 

Using \eqref{hexa1} once again, we rewrite the right hand side of \eqref{braid-relation} as 
$$
f^{-1} \si_2^{2m+1} f  (\si_1^{2m+1} f^{-1} \si_2^{2m+1} f)\, \N =  
f^{-1} \si_2^{2m+1} f (f^{-1} \si_1 \si_2 x_{12}^{-m} c^m) \N = 
$$
$$
f^{-1} \si_2^{2m} \si_2  \si_1 \si_2 x_{12}^{-m} c^m\, \N  = 
f^{-1} \si_2^{2m} \D x_{12}^{-m} c^m\, \N = f^{-1}  \D c^m \, \N.
$$
In the last step, we used the identity $\si_2 \D = \D \si_1$.

Relation \eqref{braid-relation} is proved. 
\end{proof}

If we apply both hexagon relations to the left hand side of  
\eqref{braid-relation}, then we get a useful relation on the coset $f\N$. 
Indeed, due to the calculation in \eqref{LHS-hexa1}, we have 
\begin{equation}
\label{braid-rel-hexa1}
\si_1^{2m+1} f^{-1} \si_2^{2m+1} f  \si_1^{2m+1}\, \N = f^{-1} \D c^m \, \N.
\end{equation}

On the other hand, applying \eqref{hexa11} and the identity $\si_1 \D = \D \si_2$, we get 
$$
\si_1^{2m+1} (f^{-1} \si_2^{2m+1} f  \si_1^{2m+1})\, \N = 
\si_1^{2m+1}  \si_2 \si_1 c^m x_{23}^{-m} f\, \N = 
\si_1^{2m} \D c^m x_{23}^{-m} f\, \N  = \D f c^m \, \N.
$$

Comparing this result with \eqref{braid-rel-hexa1}, we conclude that
$\D f \, \N = f^{-1} \D\, \N$. Thus, using \eqref{conj-by-D}, we see that we proved 
the following statement:
\begin{prop}  
\label{prop:H-I}
Let $\N \in \NFI_{\PB_3}(\B_3)$. If a pair $(m,f) \in \ZZ \times \F_2$ satisfies 
hexagon relations \eqref{hexa1} and \eqref{hexa11} (modulo $\N$) then
\begin{equation}
\label{H-I-here}
f \te(f) \in \N,
\end{equation}
where $\te$ is the automorphism of $\F_2$ defined in \eqref{theta}. \qed
\end{prop}  
Relation \eqref{H-I-here} can also be written in the form $f(x,y) f(y,x) \in \N$.

Let $(m,f) \in \ZZ \times [\F_2,\F_2]$ and $\N \in \NFI_{\PB_3}(\B_3)$. 
It turns out that, hexagon relations \eqref{hexa1}, \eqref{hexa11} for $(m,f)$ 
(modulo $\N$) are equivalent to somewhat simpler relations. The following 
proposition establishes this equivalence.
\begin{prop} 
\label{prop:simple-hexa}
Let $\N \in \NFI_{\PB_3}(\B_3)$ and  
$\te$ and $\tau$ be the automorphisms of $\F_2$ defined in \eqref{theta} and 
\eqref{tau}, respectively. 
A pair $(m,f) \in \ZZ \times [\F_2,\F_2]$ satisfies hexagon relations 
\eqref{hexa1}, \eqref{hexa11} (modulo $\N$) if and only if
\begin{equation} 
\label{shexagon1}
f \te(f) \in \N_{\F_2}
\end{equation}
and
\begin{equation}
 \label{shexagon11}
\tau^2(y^m f) \tau(y^m f) y^mf \in \N_{\F_2}\,.
\end{equation}
\end{prop}
\begin{proof} For our purposes, it is convenient to rewrite \eqref{shexagon1}
and \eqref{shexagon11} in the form 
\begin{equation} 
\label{simple-hexa1}
f(x,y) f(y,x) \in \N_{\F_2}
\end{equation}
and
\begin{equation}
\label{simple-hexa11}
x^m f(z,x) z^m f(y,z) y^m f \in \N_{\F_2}\,,
\end{equation}
where $z := y^{-1} x^{-1}$.

Using identities \eqref{conj-by-si1}, \eqref{conj-by-si2} and the 
property $f \in [\F_2, \F_2]$, one can prove that \eqref{hexa1} is equivalent to 
\begin{equation}
\label{hexa1-equiv}
x^m f(z, x) z^m f^{-1}(z, y) y^m f   ~\in~ \N_{\F_2}
\end{equation}
and \eqref{hexa11} is equivalent\footnote{For this equivalence, we also need \eqref{conj-by-D}.} to 
\begin{equation}
\label{hexa11-equiv}
x^m f^{-1}(x, z) z^m f(y, z) y^m f^{-1}(y,x)   ~\in~ \N_{\F_2}\,,
\end{equation}
where
$x:= x_{12},~ y := x_{23}, ~ z:= x_{23}^{-1} x_{12}^{-1}$. 

Moreover, conjugating \eqref{simple-hexa1} with $\si_1 \si_2$ and with $(\si_1 \si_2)^2$, and 
using the property $f \in [\F_2, \F_2]$ once again, we see that 
\begin{equation}
\label{shexagon1-yz}
f(z,y) \N_{\F_2} = f^{-1}(y,z) \N_{\F_2}
\end{equation}
and 
\begin{equation}
\label{shexagon1-zx}
f(x,z) \N_{\F_2} = f^{-1}(z,x) \N_{\F_2}.
\end{equation}

Let us assume that equations \eqref{hexa1} and \eqref{hexa11} are satisfied. 
Due to Proposition \ref{prop:H-I}, relation \eqref{simple-hexa1} is satisfied. 
Hence relation \eqref{shexagon1-yz} also holds.

Combining \eqref{hexa1-equiv} with \eqref{shexagon1-yz}, we conclude that 
\eqref{simple-hexa11} is satisfied.  

Let us now assume that \eqref{simple-hexa1} and \eqref{simple-hexa11} are satisfied.
Relation \eqref{simple-hexa1} implies \eqref{shexagon1-yz} and \eqref{shexagon1-zx}. 

Combining \eqref{simple-hexa1} with \eqref{simple-hexa11}, \eqref{shexagon1-yz} and 
\eqref{shexagon1-zx}, we conclude that \eqref{hexa1-equiv} and \eqref{hexa11-equiv} 
are satisfied. 

Since \eqref{hexa1-equiv} and \eqref{hexa11-equiv} are equivalent to \eqref{hexa1} and \eqref{hexa11}, 
the desired statement is proved. 
\end{proof}

We call \eqref{shexagon1}, \eqref{shexagon11} the \e{simplified hexagon relations}.
(See also \cite[Proposition 2.6]{JXthesis}.)

\bigskip

Let us denote by $\rho$ the standard homomorphism $\B_3 \to S_3$: 
$\rho(\si_1) := (1,2)$, $\rho(\si_2):= (2,3)$. Since $\N \le \PB_3$, the formula 
$\rho_{\N} (w \N):= \rho(w)$ defines the group homomorphism 
\begin{equation}
\label{rho-N}
\rho_{\N} : \B_3/\N \to S_3.
\end{equation}
It is easy to see that, for every $\N \in \NFI_{\PB_3}(\B_3)$ and 
$[m,f]\in \GT_{pr}(\N)$, 
\begin{equation}
\label{rho-N-Tmf}
\rho_{\N} \circ T_{m,f} = \rho. 
\end{equation}

Hence $T_{m,f}(\PB_3) \subset \PB_3/\N$. We set  
$$
T_{m,f}^{\PB_3} : = T_{m,f} \big|_{\PB_3} : \PB_3 \to \PB_3/\N 
$$
and notice that $\ker(T_{m,f}) = \ker(T_{m,f}^{\PB_3}) \in \NFI_{\PB_3}(\B_3)$. 

Due to the following proposition, the homomorphism $T_{m,f}^{\PB_3}$ comes 
from an endomorphism of $\PB_3$ for every $[m,f] \in \GT_{pr}(\N)$.
\begin{prop}  
\label{prop:T-m-f-PB-3}
Let $\N \in \NFI_{\PB_3}(\B_3)$ and $[m,f] \in \GT_{pr}(\N)$. Then 
\begin{equation}
\label{T-m-f-PB-3}
T_{m,f}^{\PB_3}(x_{12}) = x_{12}^{2m+1}\, \N, \qquad 
T_{m,f}^{\PB_3}(x_{23}) = f^{-1} x_{23}^{2m+1} f\, \N, \qquad
T_{m,f}^{\PB_3}(c) = c^{2m+1}\, \N. 
\end{equation}
\end{prop}  
\begin{proof}
The first two equations in \eqref{T-m-f-PB-3} are straightforward  
consequences of the definitions of $x_{12} := \si_1^2$ and 
$x_{23} := \si_2^2$. 

To prove the third equation, we will use the calculation in 
\eqref{LHS-hexa1} and relation \eqref{H-I-here}. 

Indeed, due to the calculation in \eqref{LHS-hexa1}, 
$$
T_{m,f}(\D) = f^{-1} \D c^m \, \N
$$
Hence 
$$
T^{\PB_3}_{m,f}(c) = T_{m,f} (\D^2) =  f^{-1} \D c^m  f^{-1} \D c^m \, \N = 
 \D f c^m  f^{-1} \D c^m \, \N = \D^2 c^{2m} \, \N = c^{2m+1}\, \N.
$$

Proposition \eqref{prop:T-m-f-PB-3} is proved. 
\end{proof}

\bigskip
Note that, for every $[m,f] \in \GT_{pr}(\N)$, the restriction of 
$T^{\PB_3}_{m, f}$ to $\F_2 \le \PB_3$ gives us a homomorphism 
\begin{equation}
\label{T-m-f-F-2}
T^{\F_2}_{m,f} := T^{\PB_3}_{m,f} \big|_{\F_2} : \F_2 \to \F_2/\N_{\F_2}.
\end{equation}

Let us prove that 
\begin{prop}  
\label{prop:onto}
If a pair $(m,f)\in \ZZ \times \F_2$ satisfies hexagon relations 
\eqref{hexa1} and \eqref{hexa11} and $2m+1$ represents a unit in the ring
$\ZZ/ N_{\ord}\ZZ$, then the following conditions are equivalent:

\begin{enumerate}

\item[1)] The homomorphism $T_{m,f} : \B_3 \to \B_3/ \N$ is onto.

\item[2)] The homomorphism $T^{\PB_3}_{m,f} : \PB_3 \to \PB_3/ \N$ is onto.

\item[3)] The homomorphism $T^{\F_2}_{m,f} : \F_2 \to \F_2/ \N_{\F_2}$ is onto. 

\end{enumerate}

\end{prop}  
\begin{proof}
We will start with the implication $1) \hence 2)$.  

Let $w \in \PB_3$. Since $T_{m,f}$ is onto, there exists 
$v \in \B_3$ such that $T_{m,f}(v) = w \N$. Due to \eqref{rho-N-Tmf}, 
$v \in \ker(\rho) = \PB_3$. Thus $T_{m,f}^{\PB_3}$ is indeed surjective.

Now we will take care of the implication $2) \hence 3)$.  
We will do so by showing that $x_{12} \N_{\F_2}$ and $x_{23} \N_{\F_2}$ 
belong to the image of $T_{m,f}^{\F_2}$. First, we have 
\begin{equation}
T_{m,f}^{\F_2}( x_{12} ) = x_{12}^{2m+1} \N_{\F_2}.
\end{equation}
Since $2m+1$ is coprime with the order of $x_{12} \N_{\F_2}$, 
$x_{12}^{2m+1} \N_{\F_2} \in T_{m,f}^{\F_2}(\F_2)$ implies that 
\begin{equation}
\label{x12NF2-in-the-image}
 x_{12} \N_{\F_2}  \in  T_{m,f}^{\F_2}(\F_2).
\end{equation}

Similarly, since $2m+1$ is coprime with 
$\ord\big(x_{23} \N_{\F_2} \big) =  \ord \big(f^{-1} x_{23} f\N_{\F_2}\big)$ 
and 
$$
T_{m,f}^{\F_2}(x_{23}) = f^{-1} x_{23}^{2m+1} f\, \N_{\F_2} = \big(\, f^{-1} x_{23} f \, \N_{\F_2} \,\big)^{2m+1},
$$
we conclude that 
\begin{equation}
\label{conj-x23NF2-in-the-image}
f^{-1} x_{23} f \N_{\F_2}  =   T_{m,f}^{\F_2}(x^k_{23})
\end{equation}
for some integer $k$. 

Since $T_{m,f}^{\PB_3}$ is onto, there exists $w \in \PB_3$ such that
$T_{m,f}^{\PB_3}(w) = f \N$.
Moreover $\PB_3 = \F_2 \times \lan c \ran$, so $w = \ti{w} c^j$ for some $\ti{w} \in \F_2$ and 
some integer $j$. Thus we get
\begin{equation}
\label{Tmf-PB3-tilde-w}
T_{m,f}^{\PB_3}(\ti{w}) = c^{-j(2m+1)} f \N.
\end{equation}

Since $c \in \cZ(\PB_3)$, equations \eqref{conj-x23NF2-in-the-image}
and \eqref{Tmf-PB3-tilde-w} imply that 
$$
T_{m,f}^{\PB_3}(\ti{w} x_{23}^k \ti{w}^{-1}) = c^{-j(2m+1)} f (f^{-1} x_{23} f) f^{-1} c^{j(2m+1)} \N = x_{23} \N.
$$
Note that $T_{m,f}^{\F_2} : \F_2 \to \F_2/\N_{\F_2}$ is the restriction of 
$T_{m,f}^{\PB_3}$ to $\F_2 \le \PB_3$. Therefore
\begin{equation}
\label{x23NF2-in-the-image}
x_{23} \N_{\F_2} ~\in~ T_{m,f}^{\F_2}( \F_2 ).  
\end{equation}

Combining \eqref{x12NF2-in-the-image} and \eqref{x23NF2-in-the-image}, 
we see that $\displaystyle \F_2 \overset{ T_{m,f}^{\F_2} }{\tto} \F_2/\N_{\F_2}$
is indeed surjective, i.e. the implication $2) \hence 3)$ is proved.  

Let us now prove the implication $3) \hence 1)$. 


Using $\gcd(2m+1, \ord(x_{12} N)) = \gcd(2m+1, \ord(x_{23} N)) = 1 $ and 
$2 \nmid (2 m + 1)$, it is easy to show that 
\begin{equation}
\label{gcd-2m1-siN-siN}
\gcd(2m+1, \ord(\si_1 N)) = \gcd(2m+1, \ord(\si_2 N)) = 1.
\end{equation}

Combining \eqref{gcd-2m1-siN-siN} with 
$$
T_{m,f}(\si_1) = \si_1^{2m+1} \N, 
\qquad
T_{m,f}(\si_2) = f^{-1} \si_2^{2m+1} f \N =  (f^{-1} \si_2 f \N)^{2m+1}\,,
$$
we conclude that
\begin{equation}
\label{si1N-in-the-image}
\si_1\, \N  ~\in~ T_{m,f} (\B_3)
\end{equation}
and 
\begin{equation}
\label{f-inv-si2N-f-in-the-image}
f^{-1} \si_2 f \,\N ~~\in~~ T_{m,f} (\B_3). 
\end{equation}

Surjectivity of $T_{m,f}^{\F_2}$ implies that $f \N_{\F_2} = T_{m,f}^{\F_2} (w)$
for some $w \in \F_2$. Hence
\begin{equation}
\label{Tmf-w-fN}
T_{m,f} (w) = f \N.
\end{equation}

Using \eqref{f-inv-si2N-f-in-the-image} and \eqref{Tmf-w-fN}, it is easy to see that 
\begin{equation}
\label{si2N-in-the-image}
\si_2\, \N  ~\in~ T_{m,f} (\B_3).
\end{equation}

Combining \eqref{si1N-in-the-image} and \eqref{si2N-in-the-image}, we conclude 
that $\displaystyle \B_3 \overset{ T_{m,f} }{\tto} \B_3/\N$ is indeed surjective, i.e. 
the implication $3) \hence 1)$ is also proved. 

Proposition \ref{prop:onto} is proved.  
\end{proof}

\bigskip
\begin{defi}  
\label{dfn:GT-shadows}
Let $\N \in \NFI_{\PB_3}(\B_3)$. A charming $\GT$-pair 
$[m,f] \in \GT_{pr}(\N)$ is called a $\GT$-\e{shadow with the target} $\N$ if 
the pair $(m,f)$ satisfies one of the three equivalent conditions of 
Proposition \ref{prop:onto}. We denote by $\GT(\N)$ the set of $\GT$-shadows 
with the target $\N$.
\end{defi}  

\bigskip

Using \eqref{rho-N-Tmf}, it is easy to show that, for every $[m,f] \in \GT(\N)$, 
the kernel $\K$ of the homomorphism 
$T_{m,f}: \B_3 \to \B_3 / \N$ belongs to $\NFI_{\PB_3}(\B_3)$, and
\begin{equation}
\label{ker-Tmf-ker-Tmf-PB3}
\K = \ker\big(\PB_3 \overset{T^{\PB_3}_{m,f}}{\tto} \PB_3/\N \big).
\end{equation}
Moreover, the surjectivity of $T_{m,f}$ implies that it factors as follows 
\begin{equation}
\label{Tmf-isom}
T_{m,f} = T^{\isom}_{m,f} \circ \cP_{\K}, 
\end{equation}
where $\cP_{\K}$ is the standard onto homomorphism
$\B_3 \to \B_3/\K$ and $T^{\isom}_{m,f}$ is the isomorphism  $\B_3/\K \iso \B_3/\N$
defined by the formula $T^{\isom}_{m,f}(w \K) := T_{m,f}(w)$.

Using \eqref{ker-Tmf-ker-Tmf-PB3}, it is easy to prove that, 
for every $[m,f] \in \GT(\N)$, 
\begin{equation}
\label{ker-Tmf-F2-K-F2}
\ker\big(\F_2 \overset{T^{\F_2}_{m,f}}{\tto} \F_2/\N_{\F_2} \big) = \K_{\F_2}\,,
\end{equation}
where $\K : = \ker(T_{m,f})$. 

Using \eqref{ker-Tmf-ker-Tmf-PB3} and \eqref{ker-Tmf-F2-K-F2}, we get 
the similar factorizations for the homomorphisms $T^{\PB_3}_{m,f}: \PB_3 \to \PB_3 / \N$
and for $T^{\F_2}_{m,f}: \F_2 \to \F_2 / \N_{\F_2}$, i.e. 
\begin{equation}
\label{Tmf-PB3-isom}
T^{\PB_3}_{m,f} = T^{\PB_3, \isom}_{m,f} \circ \cP_{\K}, 
\end{equation}
and 
\begin{equation}
\label{Tmf-F2-isom}
T^{\F_2}_{m,f} = T^{\F_2, \isom}_{m,f} \circ \cP_{\K_{\F_2}}, 
\end{equation}
where $T^{\PB_3, \isom}_{m,f}$ (resp. $T^{\F_2, \isom}_{m,f}$) is an isomorphism 
$\PB_3/\K \iso \PB_3/\N$ (resp. $\F_2/\K_{\F_2} \iso \F_2/\N_{\F_2}$). For example, 
the isomorphism $T_{m,f}^{\F_2, \isom} : \F_2/\K_{\F_2} \iso \F_2/\N_{\F_2}$ is defined 
by the formula:
\begin{equation}
\label{Tmf-F2-isom-dfn}
T_{m,f}^{\F_2, \isom}(w \K_{\F_2}) := T^{\F_2}_{m, f}(w).
\end{equation}

Thus we proved the first three statements of the following proposition:  
\begin{prop}  
\label{prop:K-ord-N-ord}
Let $\K, \N \in \NFI_{\PB_3}(\B_3)$. If there exists $[m,f] \in \GT(\N)$ such 
that $\K = \ker(T_{m,f})$, then 
\begin{itemize}

\item[1)] the finite groups $\B_3/\K$ and  $\B_3/\N$ are isomorphic,

\item[2)] the finite groups $\PB_3/\K$ and  $\PB_3/\N$ are isomorphic,

\item[3)] the finite groups $\F_2/\K_{\F_2}$ and  $\F_2/\N_{\F_2}$ are isomorphic and, finally, 

\item[4)] $K_{\ord} = N_{\ord}$.  

\end{itemize}
\end{prop}  
\begin{proof} 
It remains to prove that $K_{\ord} = N_{\ord}$. 

Since $2m+1$ is coprime with the orders of $x_{12}\N$, $x_{23}\N$, and $c \N$, 
we have
\begin{equation}
\label{orders-x-x-c}
\ord(x_{12}^{2m+1} \N) = \ord(x_{12}\N), 
\quad
\ord(x_{23}^{2m+1} \N) = \ord( x_{23} \N), 
\quad 
\ord(c^{2m+1} \N) = \ord(c \N).
\end{equation}
Note that $\ord(x_{23}^{2m+1} \N) = \ord(f^{-1} x_{23}^{2m+1} f \N)$. Combining this observation 
with the second equation in \eqref{orders-x-x-c}, we conclude that 
\begin{equation}
\label{orders-x23-conj-x23}
\ord(f^{-1} x_{23}^{2m+1} f \N) = \ord( x_{23} \N).
\end{equation}

Since
$$
T_{m,f}^{\PB_3, \isom}(x_{12} \K) = x_{12}^{2m+1} \N, 
\qquad 
T_{m,f}^{\PB_3, \isom}(c \K) = c^{2m+1} \N,
$$
$$
T_{m,f}^{\PB_3, \isom}(x_{23} \K) = f^{-1} x_{23}^{2m+1} f \N, 
$$
and $T_{m,f}^{\PB_3, \isom}$ is an isomorphism,  
equations \eqref{orders-x-x-c} and \eqref{orders-x23-conj-x23} imply that 
$$
\ord(x_{12} \K) = \ord(x_{12}\N), 
\quad
\ord(x_{23} \K) = \ord(x_{23}\N), 
\quad
\ord(c \K) = \ord(c \N).
$$
Thus, $K_{\ord} = N_{\ord}$. 
\end{proof}

\bigskip

Our next goal is to show that $\GT$-shadows form a groupoid $\GTSh$ 
with $\Ob(\GTSh):= \NFI_{\PB_3}(\B_3)$ and 
\begin{equation}
\label{GTSh-K-N}
\GTSh(\K, \N) := \big\{ [m,f] \in \GT(\N) ~|~ \ker(T_{m,f}) = \K \big\}, 
\qquad \K, \N \in  \NFI_{\PB_3}(\B_3).
\end{equation}

To define the composition of morphisms, we need an auxiliary construction. 

For every pair $(m,f) \in \ZZ \times \F_2$, the formulas 
\begin{equation}
\label{Emf}
E_{m,f}(x) := x^{2m+1}, \qquad 
E_{m,f}(y) := f^{-1} y^{2m+1} f
\end{equation}
define an endomorphism $E_{m,f}$ of $\F_2$. 

A direct computation shows that 
\begin{equation}
\label{E-composition}
E_{m_1,f_1} \circ  E_{m_2, f_2} = E_{m,f}\,, 
\end{equation}
where 
$$
m := 2 m_1 m_2 + m_1 + m_2, \qquad  f := f_1 E_{m_1,f_1}(f_2). 
$$

It is not hard to see\footnote{A detailed proof is given in \cite[Proposition 2.11]{JXthesis}.}
that the set $\ZZ \times \F_2$ is a monoid with respect to the binary operation
\begin{equation}
\label{bullet-Z-F2}
(m_1, f_1) \bullet (m_2, f_2) := \big(2m_1 m_2 + m_1 + m_2 , f_1 E_{m_1,f_1}(f_2) \big)
\end{equation}
and the identity element $(0, 1_{\F_2})$.
Moreover, the assignment $(m,f) \mapsto E_{m,f}$ defines a 
homomorphism of monoids $(\ZZ \times \F_2, \bullet) \to \End(\F_2)$.

Note that, if $(m,f) \in \ZZ \times \F_2$ represents a $\GT$-pair
with the target $\N \in \NFI_{\PB_3}(\B_3)$, then 
\begin{equation}
\label{E-T-F2}
T^{\F_2}_{m,f}(w) =  E_{m,f}(w) \N_{\F_2}\,, \qquad \forall~ w \in \F_2\,,
\end{equation}
where $T^{\F_2}_{m,f}$ is defined in \eqref{T-m-f-F-2}.

Let us prove the following auxiliary statement: 
\begin{prop}  
\label{prop:composition}
Let  $\N^{(1)}, \N^{(2)}, \N^{(3)} \in \NFI_{\PB_3}(\B_3)$, 
$[m_1, f_1] \in \GTSh(\N^{(2)}, \N^{(1)})$,  $[m_2, f_2] \in \GT(\N^{(3)}, \N^{(2)})$
and $N_{\ord}:= N^{(1)}_{\ord} = N^{(2)}_{\ord} = N^{(3)}_{\ord}$. If
\begin{equation}
\label{composition}
m := 2 m_1 m_2 + m_1 + m_2, 
\qquad
f := f_1 E_{m_1,f_1}(f_2),
\end{equation}
then 
\begin{equation}
\label{result-in-GTSh-3-1}
(m + N_{\ord}\ZZ, f \N^{(1)}_{\F_2}) ~\in~ \GTSh(\N^{(3)}, \N^{(1)}).
\end{equation}
The pair $[m,f] := (m + N_{\ord}\ZZ, f \N^{(1)}_{\F_2})$ depends only 
on the cosets $f_1\N^{(1)}$, $f_2 \N^{(2)}$ and residue classes 
$m_1 + N_{\ord}\ZZ$, $m_2 + N_{\ord}\ZZ$.
Moreover, the diagram 
\begin{equation}
\label{composition-diag}
\begin{tikzpicture}
\matrix (m) [matrix of math nodes, row sep=3em, column sep=4em]
{\B_3    &  \B_3 & ~ \\
 \B_3/\N^{(3)}  &  \B_3/\N^{(2)} &  \B_3/\N^{(1)} \\};
\path[->, font=\scriptsize]
(m-1-1) edge node[above] {$~~~~T_{m_2, f_2}$} (m-2-2)  
edge node[left] {$\cP_{\N^{(3)}}$} (m-2-1) 
(m-2-1) edge node[above] {$T^{\isom}_{m_2, f_2}$} (m-2-2)  
(m-1-1) edge node[above] {$~~~~T_{m_2, f_2}$} (m-2-2)  
(m-1-2) edge node[above] {$~~~~T_{m_1, f_1}$} (m-2-3)  
edge node[right] {$\cP_{\N^{(2)}}$} (m-2-2) 
(m-2-2) edge node[above] {$T^{\isom}_{m_1, f_1}$} (m-2-3)  
(m-1-1) edge[bend left = 60]  node[above] {$T_{m, f}$} (m-2-3) 
(m-2-1) edge[bend right = 30]  node[below] {$T^{\isom}_{m, f}$} (m-2-3) ;
\end{tikzpicture}
\end{equation}
commutes. In particular, 
\begin{equation}
\label{pair-to-T-functor}
T^{\isom}_{m_1, f_1} \circ T^{\isom}_{m_2, f_2} =  T^{\isom}_{m, f}\,.
\end{equation}
\end{prop}  
\begin{proof} 
The first equation in \eqref{composition} implies that
\begin{equation}
\label{2m+1}
2 m +1 = (2 m_1+1) (2 m_2+1). 
\end{equation}

Our first goal is to show that the pair $(m,f)$ satisfies hexagon relations 
\eqref{hexa1}, \eqref{hexa11} (modulo $\N^{(1)}$).

The first hexagon relation for $(m_2, f_2)$ (modulo $\N^{(2)}$) reads
\begin{equation}
\label{hexa1-mod-N2}
\si_1^{2m_2+1}  \, f_2^{-1} \si_2^{2m_2+1} f_2 \, \N^{(2)} ~ = ~ 
f_2^{-1} \si_1 \si_2 x_{12}^{-m_2} c^{m_2} \, \N^{(2)}\,.
\end{equation}

Applying $T^{\isom}_{m_1, f_1}$ to the left hand side of 
\eqref{hexa1-mod-N2}
and using \eqref{E-T-F2}, \eqref{2m+1}, we get
$$
\si_1^{(2 m_1+1)(2m_2+1)}\,  E_{m_1, f_1}(f_2)^{-1} f_1^{-1} 
\si_2^{(2 m_1+1)(2m_2+1)}   f_1  E_{m_1, f_1} (f_2) \, \N^{(1)}  = 
$$
\begin{equation}
\label{LHS-is}
\si_1^{(2 m_1+1)(2m_2+1)} \,  f^{-1} 
\si_2^{(2 m_1+1)(2m_2+1)} f \, \N^{(1)} =  
\si_1^{2 m+1}   f^{-1} \si_2^{2m+1} f \, \N^{(1)}. 
\end{equation}

Applying $T^{\isom}_{m_1, f_1}$ to the right hand side of  \eqref{hexa1-mod-N2}, 
using \eqref{T-m-f-PB-3}, 
\eqref{E-T-F2}, and hexagon relation \eqref{hexa1} for $(m_1, f_1)$, we get 
$$
E_{m_1, f_1}(f_2)^{-1} \,
( \si_1^{2m_1+1}  f_1^{-1} \si_2^{2 m_1+1} f_1) \,
x_{12}^{-m_2 (2 m_1+1)} c^{m_2 (2m_1+1)} \, \N^{(1)} = 
$$
$$
E_{m_1, f_1}(f_2)^{-1}  f_1^{-1} \si_1 \si_2
x_{12}^{-m_1} c^{m_1} x_{12}^{-m_2 (2 m_1+1)} c^{m_2 (2m_1+1)} \, \N^{(1)}  ~=~
f^{-1} \si_1 \si_2 x_{12}^{-m} c^{m} \N^{(1)}\,.
$$

Combining this result with the final expression in \eqref{LHS-is}, we see that the pair $(m,f)$
satisfies \eqref{hexa1} modulo $\N^{(1)}$.  

Applying $T^{\isom}_{m_1, f_1}$ to both sides of the second hexagon relation for 
$(m_2, f_2)$ and performing similar calculations, we see that the pair $(m,f)$ satisfies 
\eqref{hexa11} modulo $\N^{(1)}$.  

Since $2m +1 = (2 m_1+1) (2 m_2+1)$ and $2 \ol{m}_1+1, 2 \ol{m}_2+1 \in \big(\ZZ/ N_{\ord} \ZZ \big)^{\times}$, 
we conclude that $2 m+1$ represents a unit in the ring $\ZZ/ N_{\ord} \ZZ$. 

We may assume, without loss of generality, that $f_1, f_2 \in [\F_2, \F_2]$. 
Hence $f:= f_1 E_{m_1, f_1}(f_2)$ also belongs to the commutator subgroup
$[\F_2, \F_2]$. 

We proved that $(m,f)$ represents a charming $\GT$-pair with the target $\N^{(1)}$.

Recall that, since the pair $(m,f)$ satisfies hexagon relations \eqref{hexa1} 
and \eqref{hexa11} (modulo $\N^{(1)}$), the formulas 
$$
T_{m,f}(\si_1) := \si_1^{2m+1} \N^{(1)}, \qquad 
T_{m,f}(\si_2) := f^{-1} \si^{2m+1} f \N^{(1)},
$$
define a group homomorphism $T_{m,f} : \B_3 \to \B_3 / \N^{(1)}$. 

To show that the pair $(m,f)$ represents a $\GT$-shadow 
with the target $\N^{(1)}$, we need to prove that the group homomorphism 
$T_{m,f} : \B_3 \to \B_3 / \N^{(1)}$ is onto. 

Applying $T^{\isom}_{m_1, f_1} \circ T_{m_2, f_2}$ to the generators 
$\si_1$ and $\si_2$ and using \eqref{2m+1}, we see that 
$$
T^{\isom}_{m_1, f_1} \circ T_{m_2, f_2}(\si_1) = T_{m,f}(\si_1),
\qquad
T^{\isom}_{m_1, f_1} \circ T_{m_2, f_2}(\si_2) = T_{m,f}(\si_2).
$$
Therefore, 
\begin{equation}
\label{T-isom-T-is-T}
T^{\isom}_{m_1, f_1} \circ T_{m_2, f_2} = T_{m,f}\,.
\end{equation}
Hence $T_{m,f}$ is onto. Thus the pair $(m,f)$ indeed represents 
a $\GT$-shadow with the target $\N^{(1)}$. 

Combining identity \eqref{T-isom-T-is-T} with $\N^{(3)} = \ker(T_{m_2, f_2})$, we 
conclude that $\ker(T_{m,f}) = \N^{(3)}$. Hence, $T_{m,f}$ factors as
$$
T_{m,f} = T^{\isom}_{m,f} \circ \cP_{\N^{(3)}}, 
$$
where $T^{\isom}_{m,f}$ is the isomorphism $\B_3/ \N^{(3)} \iso \B_3/ \N^{(1)}$ defined 
by the formula $T^{\isom}_{m,f} (w \N^{(3)}) := T_{m,f}(w)$.

We proved the first statement of the proposition (see \eqref{result-in-GTSh-3-1}).

It is clear that $m+N_{\ord}\ZZ$ depends only the residue classes of
$m_1$ and $m_2$ in $\ZZ/N_{\ord}\ZZ$. 

Let $h_1 \in \N^{(1)}_{\F_2}$ and $h_2 \in \N^{(2)}_{\F_2}$. 
It is clear that $T^{\F_2}_{m_1+ t N_{ord}, f_1 h_1} = T^{\F_2}_{m_1,f_1}$ for 
every $t \in \ZZ$. 
Due to \eqref{E-T-F2} and $\ker(T^{\F_2}_{m_1,f_1}) = \N^{(2)}_{\F_2}$, we have
$E_{m_1, f_1}(h_2) \in \N^{(1)}_{\F_2}$. Hence 
$$
f_1 h_1 E_{m_1,f_1}(f_2 h_2) \N^{(1)}_{\F_2} = f_1 E_{m_1,f_1}(f_2) \N^{(1)}_{\F_2} = f \N^{(1)}_{\F_2}. 
$$
We proved that the $\GT$-shadow $[m,f] \in \GT(\N^{(1)})$ depends only on 
the cosets $f_1\N^{(1)}$, $f_2 \N^{(2)}$ and residue classes 
$m_1 + N_{\ord}\ZZ$, $m_2 + N_{\ord}\ZZ$.
 
It should now be clear that diagram \eqref{composition-diag} commutes. 
Indeed, the inner ``straight'' triangles commute by definition of $T^{\isom}_{m_1,f_1}$
and $T^{\isom}_{m_2, f_2}$ (see equation \eqref{Tmf-isom}). 

The triangle with the vertices $\B_3$, $\B_3/ \N^{(2)}$, $\B_3/ \N^{(1)}$ 
and the ``curved'' arrow $T_{m,f}$ commutes
due to identity \eqref{T-isom-T-is-T}. 

The definition of $T^{\isom}_{m,f}$ gives us the commutativity of the outer ``curved'' triangle
(i.e. the triangle with the vertices $\B_3$,  $\B_3/ \N^{(3)}$ and  $\B_3/ \N^{(1)}$). Combining the 
commutativity of  the outer ``curved'' triangle with identity \eqref{T-isom-T-is-T}, we conclude
that the lower ``curved'' triangle also commutes. 
 
Proposition \ref{prop:composition} is proved.
\end{proof}

We are now ready to prove that $\GTSh$ is indeed a groupoid. 
\begin{thm}  
\label{thm:GTSh}
Let  $\N^{(1)}, \N^{(2)}, \N^{(3)} \in \NFI_{\PB_3}(\B_3)$, 
$[m_1, f_1] \in \GTSh(\N^{(2)}, \N^{(1)})$,  $[m_2, f_2] \in \GT(\N^{(3)}, \N^{(2)})$
and $N_{\ord}:= N^{(1)}_{\ord} = N^{(2)}_{\ord} = N^{(3)}_{\ord}$. The formula 
\begin{equation}
\label{GTSh-composition}
[m_1, f_1] \circ [m_2,  f_2]  := [2m_1 m_2 + m_1 +m_2 , f_1 E_{m_1, f_1}(f_2)]
\end{equation}
defines a composition of morphisms in $\GTSh$. For every $\N \in \NFI_{\PB_3}(\B_3)$, the pair 
$(0, 1_{\F_2})$ represents the identity morphism in $\GTSh(\N, \N)$.
Finally, for every $[m,f] \in \GTSh(\K, \N)$, the formulas
\begin{equation}
\label{inverse}
\ti{m} + N_{\ord} \ZZ := - (2\ol{m}+1)^{-1} \ol{m}, 
\qquad
\ti{f} \K_{\F_2} := (T^{\F_2, \isom}_{m,f})^{-1}\big( f^{-1} \N_{\F_2} \big) 
\end{equation}
define the inverse $[\ti{m}, \ti{f}] \in \GTSh(\N, \K)$ of the morphism $[m,f]$.
\end{thm}  
\begin{proof}
Due to Proposition \ref{prop:composition}, formula \eqref{GTSh-composition} indeed 
defines a map 
$$
\GTSh(\N^{(2)}, \N^{(1)}) \times \GT(\N^{(3)}, \N^{(2)}) \to 
  \GT(\N^{(3)}, \N^{(1)}). 
$$

Since the binary operation $\bullet$ on $\ZZ \times \F_2$ defined in \eqref{bullet-Z-F2}
is associative, the composition of morphisms in $\GTSh$ is also associative. 

It is easy to see that the pair $(0,1_{\F_2})$ represents a $\GT$-shadow 
in $\GTSh(\N, \N)$ for every $\N \in \NFI_{\PB_3}(\B_3)$.
Moreover, since $(0,1_{\F_2})$ is the identity element of the monoid  
$(\ZZ \times \F_2, \bullet)$, $[0,1_{\F_2}]$ is indeed the identity morphism 
in $\GTSh(\N, \N)$ for every $\N \in \NFI_{\PB_3}(\B_3)$.

To take care of the inverse, we start with $[m,f] \in  \GTSh(\K, \N)$ 
and assume that the pair 
$(\ti{m} + K_{\ord} \ZZ, \ti{f}\K_{\F_2}) \in \ZZ/K_{\ord}\ZZ \times \F_2 /\K_{\F_2}$ 
is given by the formulas\footnote{Since $\GTSh(\K, \N)$ is  non-empty, $K_{\ord} = N_{\ord}$.}
\eqref{inverse}. We denote by $\ti{m}$ (resp. $\ti{f}$) any representative of the coset 
$- (2\ol{m}+1)^{-1} \ol{m}$ (resp. the coset 
$(T^{\F_2, \isom}_{m,f})^{-1}\big( f^{-1} \N_{\F_2} \big)$) in 
$\ZZ/N_{\ord}\ZZ$ (resp. in $\F_2 /\K_{\F_2}$). 

The equations in \eqref{inverse} are equivalent to 
\begin{equation}
\label{inverse-better}
2m\ti{m} + \ti{m} + m \equiv 0 \mod  N_{\ord}\,,
\qquad
T^{\F_2, \isom}_{m,f} (\ti{f} \K_{\F_2}) := f^{-1} \N_{\F_2}\,. 
\end{equation}

The first equation in \eqref{inverse-better} implies that 
\begin{equation}
\label{2m+1-inv}
(2m+1) (2\ti{m}+1) \equiv 1  \mod 2 N_{\ord}\,.
\end{equation}
Hence $2 \ti{m}+1$ represents a unit in $\ZZ/N_{\ord}\ZZ$. 

Since 
$$
\si_1^{2 N_{\ord}}, ~\si_2^{2 N_{\ord}}  \in \N,  
$$
identity \eqref{2m+1-inv} implies that 
\begin{equation}
\label{si1-si2-powers}
\si_1^{(2m+1) (2\ti{m}+1)} \N = \si_1 \N, 
\qquad 
\si_2^{(2m+1) (2\ti{m}+1)} \N = \si_2 \N.
\end{equation}

Since $f^{-1} \N_{\F_2}$ belongs to $[\F_2/\N_{\F_2} , \F_2/\N_{\F_2}]$, so does 
$\ti{f} \K_{\F_2}$. 

Let us prove that the pair $(\ti{m}, \ti{f})$ satisfies 
\eqref{hexa1} and \eqref{hexa11} (modulo $\K$).  

Applying $T^{\isom}_{m,f}$ to 
$\ti{f}^{-1} \si_2^{2\ti{m}+1} \ti{f} \si_1^{2\ti{m}+1} \K$ and 
using the second equation in \eqref{inverse-better} and 
identities \eqref{si1-si2-powers},
we get
$$
T^{\isom}_{m,f} \big(
\ti{f}^{-1} \si_2^{2\ti{m}+1} \ti{f} \si_1^{2\ti{m}+1} \K
\big) = 
f f^{-1} \si_2^{(2m+1)(2\ti{m}+1)} f f^{-1} \si_1^{(2m+1)(2\ti{m}+1)} \N = 
\si_2 \si_1 \N.
$$

Furthermore, applying $T^{\isom}_{m,f}$ to $\si_2 \si_1 c^{\ti{m}} x_{23}^{- \ti{m}} \ti{f} \K$ and 
using hexagon relation \eqref{hexa11} for $(m,f)$ and the first equation in \eqref{inverse-better}, 
we get 
$$
T^{\isom}_{m,f}
\big( \si_2 \si_1 c^{\ti{m}} x_{23}^{- \ti{m}} \ti{f} \K \big) = 
(f^{-1} \si_2^{2m+1} f \si_1^{2m+1}) \N \, (c^{(2m+1)\ti{m}} f^{-1} x_{23}^{- (2m+1)\ti{m}} f) \N =
$$
$$
\si_2 \si_1 c^{m} x_{23}^{-m} f c^{(2m+1)\ti{m}} f^{-1} x_{23}^{- (2m+1)\ti{m}} f  f^{-1}\, \N = 
\si_2 \si_1 c^{2m\ti{m}+\ti{m}+m}  x_{23}^{-(2m\ti{m}+\ti{m}+m)} \, \N = \si_2 \si_1 \N.
$$

Since 
$$
T^{\isom}_{m,f} \big(
\ti{f}^{-1} \si_2^{2\ti{m}+1} \ti{f} \si_1^{2\ti{m}+1} \K
\big)
~ = ~
T^{\isom}_{m,f}\big( \si_2 \si_1 c^{\ti{m}} x_{23}^{- \ti{m}} \ti{f} \K \big) 
$$
and $T^{\isom}_{m,f}$ is an isomorphism, we conclude that the 
pair  $(\ti{m} , \ti{f})$ satisfies hexagon relation \eqref{hexa11}. 

Applying $T^{\isom}_{m,f}$ to both sides of
$$
\si_1^{2\ti{m}+1}  \, \ti{f}^{-1} \si_2^{2\ti{m}+1} \ti{f}\, \K ~ \overset{?}{=} ~ 
\ti{f}^{-1} \si_1 \si_2 x_{12}^{-\ti{m}} c^{\ti{m}} \, \K
$$
and performing similar calculations, we see that the pair $(\ti{m}, \ti{f})$
also satisfies hexagon relation \eqref{hexa1}. 

Using the equations in \eqref{inverse-better} we see that 
the composition
$$
T^{\isom}_{m,f} \circ T_{\ti{m}, \ti{f} } : \B_3 \to \B_3/\N
$$
coincides with the standard projection $\cP_{\N} : \B_3 \to \B_3/\N$. 
Hence the group homomorphism $ T_{\ti{m}, \ti{f}} : \B_3 \to \B_3 /\K$ is onto 
and 
$$
\ker( T_{\ti{m}, \ti{f} }) = \N.
$$

Thus we proved that 
$$
(\ti{m} + N_{\ord} \ZZ, \ti{f}\K_{\F_2}) \in \GTSh(\N, \K).
$$

The equations in \eqref{inverse-better} imply that 
$$
[m,f] \circ [\ti{m}, \ti{f}] = [0, 1_{\F_2}]. 
$$

Since 
$$
[\ti{m}, \ti{f}] \circ [m,f] =  
\big(2\ti{m}m + m + \ti{m} + N_{\ord}\ZZ \,,\, 
\ti{f} \K_{\F_2} \, T^{\F_2}_{\ti{m}, \ti{f}}(f) 
\big) = \big(N_{\ord}\ZZ \,,\, 
\ti{f} \K_{\F_2} \, T^{\F_2}_{\ti{m}, \ti{f}}(f) 
\big)
$$
it remains to prove that
\begin{equation}
\label{f-part}
\ti{f} \K_{\F_2} \, T^{\F_2}_{\ti{m}, \ti{f}}(f) \, \overset{?}{=} \, 1_{\F_2/ \K_{\F_2}}.
\end{equation}

Applying $T^{\F_2, \isom}_{m,f}$ to the left hand side of \eqref{f-part}
and using $T^{\isom}_{m,f} \circ T_{\ti{m}, \ti{f} } = \cP_{\N}$, 
we get 
$$
T^{\F_2, \isom}_{m,f}
\big(
\ti{f} \K_{\F_2} \, T^{\F_2}_{\ti{m}, \ti{f}}(f)
\big) = f^{-1} \N_{\F_2} \, f \N_{\F_2}  =  1_{\F_2/ \N_{\F_2}}.
$$

Thus, since $T^{\F_2, \isom}_{m,f}$ is an isomorphism from 
$\F_2/ \N_{\F_2}$ to $\F_2/ \K_{\F_2}$, we conclude that identity \eqref{f-part} holds.

Theorem \ref{thm:GTSh} is proved. 
\end{proof}
\begin{remark}  
\label{rem:same-indices}
Proposition \ref{prop:K-ord-N-ord} implies that, if $\GTSh(\K, \N)$ is non-empty, then 
$$
|\PB_3: \K| = |\PB_3: \N|, \qquad
|\F_2: \K_{\F_2}| = |\F_2: \N_{\F_2}|, \qquad K_{\ord}=N_{\ord}.
$$ 
\end{remark}

\subsection{The reduction map}
\label{sec:cR-N-H}

Let $\N, \H \in \NFI_{\PB_3}(\B_3)$ and $\N \le \H$. 
In the following proposition, we consider this situation and get 
a natural map $\cR_{\N, \H} : \GT(\N) \to \GT(\H)$.

\begin{prop}  
\label{prop:get-cR-N-H}
Let $\N, \H \in \NFI_{\PB_3}(\B_3)$, $\N \le \H$ and $(m,f) \in \ZZ \times \F_2$ represent 
a $\GT$-pair with the target $\N$. Then $H_{\ord} | N_{\ord}$, $\N_{\F_2} \le \H_{\F_2}$ and 

\begin{itemize}

\item[{\rm a)}] the same pair $(m,f)$ also represents an element in $\GT_{pr}(\H)$; 
moreover the resulting $\GT$-pair $[m,f] \in \GT_{pr}(\H)$ depends only 
on $(m + N_{\ord} \ZZ, f \N_{\F_2}) $; 

\item[{\rm b)}] if the $\GT$-pair $[m,f] \in \GT_{pr}(\N)$ is charming then so is the corresponding 
$\GT$-pair in $\GT_{pr}(\H)$; 

\item[{\rm c)}] if the pair $(m, f)$ represents a $\GT$-shadow with the target $\N$, then $(m, f)$ 
also represents a $\GT$-shadow with the target $\H$.
 
\end{itemize}
Let us denote by $T_{m,f,\H}$ the group homomorphism 
$\B_3 \to \B_3/\H$ corresponding to 
$[m,f] \in \GT_{pr}(\H)$. In the set-up of statement {\rm a)}, the following diagram 
\begin{equation}
\label{diag-Tmf-N-H}
\begin{tikzpicture}
\matrix (m) [matrix of math nodes, row sep=1.8em, column sep=1.8em]
{\B_3    & ~ &  \B_3/\N & ~ \\
 ~ & \B_3/\H & ~\\};
\path[->, font=\scriptsize]
(m-1-1) edge node[above] {$T_{m, f}$} (m-1-3)
edge node[left] {$T_{m, f, \H}~$} (m-2-2)
(m-1-3) edge node[right] {$~~\cP_{\N, \H}$} (m-2-2);
\end{tikzpicture}
\end{equation}
commutes.
\end{prop}  
\begin{proof} Since 
$$
\cP_{\N,\H}(x_{12} \N)= x_{12} \H, \qquad 
\cP_{\N, \H}(x_{23} \N)= x_{23} \H, \qquad
\cP_{\N, \H}(c \N)= c \H,
$$
$\ord(x_{12} \H) | \ord(x_{12} \N)$, $\ord(x_{23} \H) | \ord(x_{23} \N)$ and $\ord(c \H) | \ord(c \N)$. 
Hence $H_{\ord}$ divides $N_{\ord}$. The inclusion $\N_{\F_2} \le \H_{\F_2}$ is obvious.

~\\
a) Applying the homomorphism 
$\cP_{\N, \H}: \B_3/\N \to \B_3/\H$ to \eqref{hexa1} and \eqref{hexa11}, we see that the pair 
$(m,f)$ satisfies the hexagon relations modulo $\H$ if it satisfies the hexagon relations 
modulo $\N$. 
Thus $(m,f)$ represents an element in $\GT_{pr}(\H)$. 

It is obvious that the resulting $\GT$-pair $[m,f] \in \GT_{pr}(\H)$ depends only on the 
residue class of $m$ modulo $N_{\ord}$ and the coset $f\N_{\F_2}$. 

As above, we denote by $T_{m,f,\H}$ the group homomorphism $\B_3 \to \B_3/\H$ 
corresponding to $[m,f] \in \GT_{pr}(\H)$. 
Applying $T_{m,f,\H}$ and $\cP_{\N, \H} \circ T_{m,f}$ to the generators $\si_1, \si_2$, 
we see that the diagram in \eqref{diag-Tmf-N-H} indeed commutes. 

~\\
b) Since $2m+1$ represents a unit in $\ZZ/N_{\ord} \ZZ$, $2m+1$ also 
represents a unit in $\ZZ/H_{\ord} \ZZ$. 
Since $f \N_{\F_2}$ belongs to 
the commutator subgroup $[\F_2/\N_{\F_2}, \F_2/\N_{\F_2}]$, we have 
$$
f \H_{\F_2} \in [\F_2/\H_{\F_2}, \F_2/\H_{\F_2}]. 
$$
Thus $(m,f)$ represents a charming $\GT$-pair with the target $\H$.  

~\\
c) This statement follows easily from the commutativity of the diagram in 
\eqref{diag-Tmf-N-H} and the surjectivity of the homomorphism $\cP_{\N, \H}$. 
\end{proof}

Due to Proposition \ref{prop:get-cR-N-H}, the formula
\begin{equation}
\label{cR-N-H}
\cR_{\N, \H}([m,f]) := \big( m+ H_{\ord} \ZZ, f \H_{\F_2} \big)
\end{equation}
defines a map $\cR_{\N, \H} : \GT(\N) \to \GT(\H)$. We call $\cR_{\N, \H}$ the 
\e{reduction map}.  

Just as in \cite[Definition 3.12]{GTshadows}, we say that a $\GT$-shadows 
$[m,f] \in \GT(\H)$ \e{survives into} $\N$ if $[m,f]$ belongs to the image of $\cR_{\N, \H}$.

\subsection{Connected compsonents of the groupoid $\GTSh$ and its isolated objects}
\label{sec:isolated}

The groupoid $\GTSh$ is highly disconnected. 
Indeed, if $|\PB_3: \N| \neq |\PB_3: \K|$, then 
$\GTSh(\K, \N)$ is empty (see Remark \ref{rem:same-indices}).
 For $\N \in \NFI_{\PB_3}(\B_3)$, we denote by 
$\GTSh_{\conn}(\N)$ the connected component of $\N$ in the groupoid $\GTSh$. 
Since, for every $\N \in \NFI_{\PB_3}(\B_3)$, $\GT(\N)$ is finite, so is the groupoid 
$\GTSh_{\conn}(\N)$.

\begin{defi}  
\label{dfn:isolated}
Let  $\N \in \NFI_{\PB_3}(\B_3)$. A $\GT$-shadow $[m,f] \in \GT(\N)$ is called 
\e{settled} if $\ker(T_{m,f}) = \N$, i.e. $[m,f] \in \GTSh(\N,\N)$. An object
$\N$ of the groupoid $\GTSh$ is called \e{isolated} if every $\GT$-shadow
in $\GT(\N)$ is settled. 
\end{defi}  
It is clear that $\N \in \NFI_{\PB_3}(\B_3)$ is isolated if and only if the connected 
component of $\N$ in the groupoid $\GTSh$ has exactly one object. Of course, in this 
case, $\GT(\N)= \GTSh(\N,\N)$. In particular, $\GT(\N)$ is a group.  

\begin{prop}  
\label{prop:N-diamond}
For every $\N \in \NFI_{\PB_3}(\B_3)$, the subgroup
\begin{equation}
\label{N-diamond}
\N^{\dia} ~: =~ \bigcap_{\K \in \Ob(\GTSh_{\conn}(\N))}\, \K 
\end{equation}
is an isolated object of the groupoid $\GTSh$.
\end{prop}  
\begin{proof} Since the groupoid $\GTSh_{\conn}(\N)$ has finitely many objects and 
$\NFI_{\PB_3}(\B_3)$ is closed under finite intersections, 
$\N^{\dia}$ belongs to $\NFI_{\PB_3}(\B_3)$. 

To prove that $\N^{\dia}$ is isolated, we consider $[m,f] \in \GT(\N^{\dia})$
and $\K \in \Ob(\GTSh_{\conn}(\N))$.

Since $\N^{\dia} \le \K$, Proposition \ref{prop:get-cR-N-H} implies that
the pair $(m,f)$ also represents a $\GT$-shadow with the target 
$\K$. Just as in Proposition \ref{prop:get-cR-N-H}, we denote 
by $T_{m,f,\K}$ the group homomorphism $\B_3 \to \B_3/\K$ corresponding 
to the $\GT$-shadow $[m,f] \in \GT(\K)$. Let us also recall that 
\begin{equation}
\label{Tmf-K-cP-Tmf}
T_{m,f,\K} = \cP_{\N^{\dia}, \K} \circ T_{m,f}\,.
\end{equation}

Let $w \in \N^{\dia}$. Since $w \in \H$ for every $\H \in \Ob(\GTSh_{\conn}(\N))$, we have 
$$
w \in \ker(T_{m, f,\K})
$$
Let $w^{\dia} \in \B_3$ be a representative of the coset 
$T_{m,f}(w) \in \B_3/\N^{\dia}$. Using \eqref{Tmf-K-cP-Tmf} we conclude that
$w^{\dia} \in \K$ for every $\K \in \Ob(\GTSh_N^{\conn})$. Therefore $w^{\dia} \in \N^{\dia}$ 
and hence $w \in \ker(\B_3 \overset{T_{m,f}}{\tto} \B_3/\N^{\dia})$. 

We proved that $\N^{\dia} \le \ti{\K}$, where $\ti{\K}:=\ker(\B_3 \overset{T_{m,f}}{\tto} \B_3/\N^{\dia})$.
Since $|\B_3 : \tilde{\K}|=|\B_3 : \N^{\dia}|$ (see Proposition \ref{prop:K-ord-N-ord}) and $\N^{\dia}$ 
has finite index in $\B_3$, we conclude 
that $\ker(\B_3 \overset{T_{m,f}}{\tto} \B_3/\N^{\dia}) = \N^{\dia}$. 
\end{proof}

Proposition \ref{prop:N-diamond} implies that the subposet $\NFI^{isolated}_{\PB_3}(\B_3)$ of 
isolated elements in $\NFI_{\PB_3}(\B_3)$ is cofinal, i.e. for every $\N \in \NFI_{\PB_3}(\B_3)$, 
there exists $\ti{\N} \in \NFI^{isolated}_{\PB_3}(\B_3)$ such that $\ti{\N} \le \N$. 

The proof of the following proposition is straightforward and we leave it to the reader: 
\begin{prop}  
\label{prop:cap-isolated}
For all $\N, \K \in \NFI^{isolated}_{\PB_3}(\B_3)$, 
$\N \cap \K \in \NFI^{isolated}_{\PB_3}(\B_3).$
\qed
\end{prop}  

\begin{remark}  
\label{rem:cR-homomorphism}
Let $\N, \H \in \NFI^{isolated}_{\PB_3}(\B_3)$ and $\N \le \H$. Recall that, 
in this case, $\GT(\N) = \GTSh(\N, \N)$ and $\GT(\H) = \GTSh(\H, \H)$, i.e. 
$\GT(\N)$ and $\GT(\H)$ are (finite) groups. 
It is easy to see that the reduction map $\cR_{\N, \H}: \GT(\N) \to \GT(\H)$ (see \eqref{cR-N-H}) 
is a group homomorphism. 
Indeed, both $[m, f] \in \GT(\N)$ and $\cR_{\N, \H}([m,f]) \in \GT(\H)$ are represented by the same 
pair $(m,f) \in \ZZ \times \F_2$ and the composition of $\GT$-shadows is 
defined in terms of their representatives 
(see equation \eqref{GTSh-composition} in Theorem \ref{thm:GTSh}). 
If $\N, \H \in \NFI^{isolated}_{\PB_3}(\B_3)$ and $\N \le \H$, we call 
$\cR_{\N, \H}: \GT(\N) \to \GT(\H)$ the \e{reduction homomorphism}. 
\end{remark}

\section{The transformation groupoid $\GTh^{gen}_{\NFI}$ and genuine $\GT$-shadows}
\label{sec:GThat-NFI}

Let $\N \in \NFI_{\PB_3}(\B_3)$ and $(\hat{m}, \hat{f}) \in \GTh_{gen}$. 
Recall that $\hcP_{N}$ denotes the standard (continuous) group homomorphism 
from $\wh{\B}_3$ to $\B_3/ \N$ and $T_{\hat{m}, \hat{f}}$ denotes the 
continuous automorphism of $\wh{\B}_3$ defined in \eqref{Tmf-hat}. 
Let us consider the composition 
\begin{equation}
\label{hcP-Tmf-hat}
\hcP_{\N} \circ T_{\hat{m}, \hat{f}} \big|_{\B_3} : \B_3 \to \B_3/\N.
\end{equation}

Using the fact that $\B_3$ is dense in $\wh{\B}_3$, one can easily prove that 
the homomorphism \eqref{hcP-Tmf-hat} is surjective. In the following proposition, 
we use \eqref{hcP-Tmf-hat} to define a right action of $\GTh_{gen}$ on 
$\NFI_{\PB_3}(\B_3)$:
\begin{prop}  
\label{prop:action}
Let $\N \in \NFI_{\PB_3}(\B_3)$. For every $(\hat{m}, \hat{f}) \in \GTh_{gen}$,
the pair
$$
\big( \hcP_{N_{\ord}}(\hat{m}), \hcP_{\N_{\F_2}}(\hat{f}) \big)
$$
is a $\GT$-shadow with the target $\N$. Furthermore, the assignment 
\begin{equation}
\label{action-NFI}
\N^{(\hat{m}, \hat{f})} := \ker\big( \hcP_{\N} \circ T_{\hat{m}, \hat{f}} \big|_{\B_3} \big)
\end{equation}
defines a right action of $\GTh_{gen}$ on $\NFI_{\PB_3}(\B_3)$.
\end{prop}  
\begin{proof}
Let $m \in \ZZ$ (resp. $f \in \F_2$) be any representative of the residue class 
$\hcP_{N_{\ord}}(\hat{m}) \in  \ZZ / N_{\ord} \ZZ$ 
(resp. of the coset $\hcP_{\N_{\F_2}}(\hat{f}) \in \F_2/\N_{\F_2}$). 

Since the pair $(\hat{m}, \hat{f})$ satisfies \eqref{hexa1-hat} and \eqref{hexa11-hat}, 
the pair $(m,f)$ satisfies hexagon relations \eqref{hexa1} and \eqref{hexa11} modulo $\N$. 

Since $2\hat{m}+1$ is a unit in $\Zhat$, the integer $2m+1$ represents a unit in $\ZZ/N_{\ord}\ZZ$.

The property $\hat{f} \in  [\wh{\F}_2, \wh{\F}_2]^{top.\,cl.}$ implies that 
$$
f \N_{\F_2} ~\in~ [\F_2/ \N_{\F_2}, \F_2/ \N_{\F_2}].
$$

Finally, it is easy to see that the homomorphism $T_{m,f} : \B_3 \to \B_3/\N$
coincides with $\hcP_{\N} \circ T_{\hat{m}, \hat{f}}$:
\begin{equation}
\label{Tmf-Tmf-hat}
T_{m,f} = \hcP_{\N} \circ T_{\hat{m}, \hat{f}} \big|_{\B_3} \,.
\end{equation}
In particular, $T_{m,f}$ is surjective. 

We proved that the pair $(m + N_{\ord}\ZZ, f \N_{\F_2})$ is a $\GT$-shadow with the target $\N$
and 
$$
\N^{(\hat{m}, \hat{f})} = \ker(T_{m,f}). 
$$ 
Hence $\N^{(\hat{m}, \hat{f})}  \in \NFI_{\PB_3}(\B_3)$.

We say that the $\GT$-shadow $[m,f] \in \GT(\N)$ \e{comes from} the element 
$(\hat{m}, \hat{f}) \in \GTh_{gen}$.

Let us consider the following diagram:
\begin{equation}
\label{diag-Tmf-hat-Tmf}
\begin{tikzpicture}
\matrix (m) [matrix of math nodes, row sep=2.6em, column sep=2.6em]
{~ & \wh{\B}_3    &  \wh{\B}_3 \\
 \B_3 & \B_3/\K & \B_3/\N \\};
\path[->, font=\scriptsize]
(m-1-2) edge node[above] {$~~T_{\hat{m}, \hat{f}}$} (m-1-3)
edge node[right] {$\hcP_{\K}$} (m-2-2)
 (m-1-3) edge node[right] {$\hcP_{\N}$} (m-2-3)
(m-2-2) edge node[above] {$T_{m,f}^{\isom}$} (m-2-3)
(m-2-1) edge node[above] {$\cP_{\K}$} (m-2-2)
edge (m-1-2)
(m-2-1) edge[bend right = 35]  node[below] {$T_{m, f}$} (m-2-3) ;
\end{tikzpicture}
\end{equation}
where $\K: = \ker(T_{m,f})$ and the slanted straight arrow is the standard inclusion map 
$j: \B_3 \to \wh{\B}_3$. 

We claim that the diagram in \eqref{diag-Tmf-hat-Tmf} commutes.
Indeed, the outer ``curved'' rectangle commutes due to \eqref{Tmf-Tmf-hat}. The 
lower ``curved'' triangle commutes due to the identity $T_{m,f} = T^{\isom}_{m,f} \circ \cP_{\K}$.
The left triangle commutes 
by definition of $\wh{\B}_3$. Finally, the continuous maps $\hcP_{\N} \circ T_{\hat{m}, \hat{f}}$
and $T^{\isom}_{m,f} \circ \hcP_{K}$ agree on the dense subset $\B_3 \subset \wh{\B}_3$
and $\B_3/\N$ is Hausdorff. Thus the inner square in \eqref{diag-Tmf-hat-Tmf} also commutes.

It is clear that 
$$
\hcP_{\N} \circ T_{0, 1_{\wh{\F}_2}}  \big|_{\B_3} = \cP_{\N}.  
$$
Hence $\N^{( 0, 1_{\wh{\F}_2} )} = \N$. 

It remains to prove that, for all $(\hat{m}_1, \hat{f}_1),  (\hat{m}_2, \hat{f}_2) \in \GTh_{gen}$, 
\begin{equation}
\label{indeed-the-action}
\big( \N^{(\hat{m}_1, \hat{f}_1)} \big)^{(\hat{m}_2, \hat{f}_2)} = \N^{(\hat{m}, \hat{f})}\,,
\end{equation}
where $(\hat{m}, \hat{f}) := (\hat{m}_1, \hat{f}_1) \bullet (\hat{m}_2, \hat{f}_2)$. 

For this purpose, we will use the inner square of the diagram in \eqref{diag-Tmf-hat-Tmf}. 
We set $\K:= \N^{(\hat{m}_1, \hat{f}_1)} $ and $\H:= \K^{(\hat{m}_2, \hat{f}_2)}$. Then, putting together 
the ``squares'' corresponding to $(\hat{m}_1, \hat{f}_1)$ and $(\hat{m}_2, \hat{f}_2)$, 
adding the obvious ``triangle with the vertex'' $\B_3$, the 
``curved arrow'' $\hcP_{\N} \circ T_{\hat{m}, \hat{f}}\big|_{\B_3}$, and
using \eqref{Tmf-hat-functor}, we get the following commutative diagram:
\begin{equation}
\label{diag-action}
\begin{tikzpicture}
\matrix (m) [matrix of math nodes, row sep=2.6em, column sep=2.6em]
{~ & \wh{\B}_3    &  \wh{\B}_3 & \wh{\B}_3\\
 \B_3 & \B_3/\H & \B_3/\K & \B_3/\N \\};
\path[->, font=\scriptsize]
(m-1-2) edge node[above] {$~~T_{\hat{m}_2, \hat{f}_2}$} (m-1-3)
edge node[right] {$\hcP_{\H}$} (m-2-2)
(m-1-3) edge node[above] {$~~T_{\hat{m}_1, \hat{f}_1}$} (m-1-4)
edge node[right] {$\hcP_{\K}$} (m-2-3)
(m-1-4) edge node[right] {$\hcP_{\N}$} (m-2-4)
(m-2-2) edge node[above] {$T_{m_2, f_2}^{\isom}$} (m-2-3)
(m-2-3) edge node[above] {$T_{m_1, f_1}^{\isom}$} (m-2-4)
(m-2-1) edge node[above] {$\cP_{\H}$} (m-2-2) edge (m-1-2)
(m-2-1) edge[bend right = 35]  node[above] {$\hcP_{\N} \circ T_{\hat{m}, \hat{f}}\big|_{\B_3}$} (m-2-4) ;
\end{tikzpicture}
\end{equation}
where $[m_1,f_1] \in \GT(\N)$ and $[m_2,f_2] \in \GT(\K)$ are the $\GT$-shadows coming 
from $(\hat{m}_1,\hat{f}_1)$ and  $(\hat{m}_2,\hat{f}_2)$, respectively.

The commutativity of the lower ``curved rectangle'' in \eqref{diag-action}
implies that $\H = \N^{(\hat{m}, \hat{f})}$. Thus identity
\eqref{indeed-the-action} holds.
\end{proof}

\bigskip

For $\N \in \NFI_{\PB_3}(\B_3)$ and $(\hat{m}, \hat{f}) \in \GTh_{gen}$, 
we denote by $\PR_{\N}(\hat{m}, \hat{f})$ the $\GT$-shadow with the target $\N$
that comes from $(\hat{m}, \hat{f})$, i.e. 
$$
\PR_{\N}(\hat{m}, \hat{f}) := 
\big( \hcP_{N_{\ord}}(\hat{m}), \hcP_{\N_{\F_2}}(\hat{f}) \big). 
$$
In view of Corollary \ref{cor:genuine-iff} which is proved later, $\PR_{\N}(\hat{m}, \hat{f})$ is 
called the approximation of the element $(\hat{m}, \hat{f}) \in \GTh_{gen}$.

We denote by $\GTh^{gen}_{\NFI}$ the transformation groupoid of the action of 
$\GTh_{gen}$ on $\NFI_{\PB_3}(\B_3)$, i.e. $\Ob(\GTh^{gen}_{\NFI}) = \NFI_{\PB_3}(\B_3)$ 
and 
$$ 
\GTh^{gen}_{\NFI}(\K,\N) := \{(\hat{m}, \hat{f}) \in \GTh_{gen} ~|~ \N^{(\hat{m}, \hat{f})} = \K\}. 
$$

\begin{defi}  
\label{dfn:genuine}
Let $\N \in \NFI_{\PB_3}(\B_3)$ and $[m,f] \in \GT(\N)$. We say that the $\GT$-shadow 
$[m,f]$ is \e{genuine} if there exists $(\hat{m}, \hat{f}) \in \GTh_{gen}$ such that 
$[m,f]$ comes from $(\hat{m}, \hat{f})$, i.e. 
$$
m + N_{\ord}\ZZ = \hcP_{N_{\ord}}(\hat{m}), \qquad 
f \N_{\F_2} = \hcP_{\N_{\F_2}}(\hat{f}).
$$
Otherwise, the $\GT$-shadow is called \e{fake}.
\end{defi}  

Let $\N \in \NFI_{\PB_3}(\B_3)$. 
Due to Proposition \ref{prop:hat-N-kernel}, the subgroup 
$\hcP^{-1}_{\N}(1_{\B_3 / \N}) \le \wh{\B}_3$ 
(resp. $\hcP^{-1}_{\N_{\F_2}}(1_{\F_2 / \N_{\F_2}}) \le \wh{\F}_2$)
coincides with the profinite completion of $\N$ (resp. with the profinite completion 
of $\N_{\F_2}$). By abuse of 
notation, we identify $\hcP^{-1}_{\N}(1_{\B_3 / \N})$ (resp.
$\hcP^{-1}_{\N_{\F_2}}(1_{\F_2 / \N_{\F_2}})$) with $\wh{\N}$
(resp. with $\wh{\N}_{\F_2}$). 
We will need the following statement:
\begin{prop}  
\label{prop:KF2-hat-NF2-hat}
Let $\N \in  \NFI_{\PB_3}(\B_3)$ and $(\hat{m}, \hat{f}) \in \GTh_{gen}$. 
If $\K$ is the source of the $\GT$-shadow $\PR_{\N}(\hat{m}, \hat{f})$, then 
\begin{equation}
\label{Tmf-hat-N-hat}
T_{\hat{m}, \hat{f}} \big( \wh{\K} \big) = \wh{\N}
\end{equation}
and
\begin{equation}
\label{Emf-hat-NF2-hat}
E_{\hat{m}, \hat{f}} \big( \wh{\K}_{\F_2} \big) = \wh{\N}_{\F_2}.
\end{equation}
\end{prop}  
\begin{proof}
Let $(m,f) \in \ZZ \times \F_2$ be a pair that represents the $\GT$-shadow
$\PR_{\N}(\hat{m}, \hat{f})$ and $\hat{w} \in \wh{\K} = \ker(\wh{\B}_3 \overset{\hcP_{\K}}{\to} \B_3/\K)$. 
Since the diagram in \eqref{diag-Tmf-hat-Tmf} commutes, 
$$
\hcP_{\N} \circ T_{\hat{m}, \hat{f}}(\hat{w}) = 1_{\B_3/\N}. 
$$
Hence $T_{\hat{m}, \hat{f}} (\wh{\K}) \subset \wh{\N} =  \ker(\wh{\B}_3 \overset{\hcP_{\N}}{\tto} \B_3/\N)$.

Since 
$|\wh{\B}_3 : \wh{\N}| = |\B_3 : \N | = |\B_3 : \K|  = |\wh{\B}_3 : \wh{\K}|=
|\wh{\B}_3 : T_{\hat{m}, \hat{f}} (\wh{\K})|$, 
the inclusion $T_{\hat{m}, \hat{f}} (\wh{\K}) \subset \wh{\N}$ implies that 
$T_{\hat{m}, \hat{f}} (\wh{\K}) = \wh{\N}$. 

Identity \eqref{Emf-hat-NF2-hat} can be proved in a similar way using 
the commutative diagram
\begin{equation}
\label{diag-Tmf-hat-Tmf-F2}
\begin{tikzpicture}
\matrix (m) [matrix of math nodes, row sep=2.6em, column sep=2.6em]
{~ & \wh{\F}_2    &  \wh{\F}_2 \\
 \F_2 & \F_2/\K_{\F_2} & \F_2/\N_{\F_2} \\};
\path[->, font=\scriptsize]
(m-1-2) edge node[above] {$~~T_{\hat{m}, \hat{f}}$} (m-1-3)
edge node[right] {$\hcP_{\K_{\F_2}}$} (m-2-2)
 (m-1-3) edge node[right] {$\hcP_{\N_{\F_2}}$} (m-2-3)
(m-2-2) edge node[above] {$T_{m,f}^{\F_2, \isom}$} (m-2-3)
(m-2-1) edge node[above] {$~~\cP_{\K_{\F_2}}$} (m-2-2)
edge node[left] {$j$~~} (m-1-2)
(m-2-1) edge[bend right = 35]  node[below] {$T^{\F_2}_{m, f}$} (m-2-3) ;
\end{tikzpicture}
\end{equation}
where $T_{m, f}^{\F_2, \isom}$ is the isomorphism 
$\F_2/\K_{\F_2} \iso \F_2/\N_{\F_2}$ defined in \eqref{Tmf-F2-isom-dfn}.
\end{proof}

The following theorem gives us a link between $\GTh_{gen}$ and the groupoid $\GTSh$: 
\begin{thm}  
\label{thm:GThgen-GTSh}
Let $\N \in  \NFI_{\PB_3}(\B_3)$.
The assignments 
\begin{equation}
\label{PR-dfn}
\PR(\N):=\N, \qquad
\PR_{\N}(\hat{m}, \hat{f}) = \big( \hcP_{N_{\ord}}(\hat{m}),  \hcP_{\N_{\F_2}}(\hat{f}) \big)
\end{equation}
define a functor from the transformation groupoid $\GTh^{gen}_{\NFI}$ to $\GTSh$. 
\end{thm}  
\begin{proof} Let $(m,f) \in \ZZ \times [\F_2, \F_2]$ be a pair that represents 
$\big( \hcP_{N_{\ord}}(\hat{m}),  \hcP_{\N_{\F_2}}(\hat{f}) \big)$.

Due to the first statement of Proposition \ref{prop:action}, $[m,f]$ is a 
$\GT$-shadow with the target $\N$. Moreover, since 
$$
\ker(T_{m,f}) = \N^{(\hat{m}, \hat{f})}, 
$$ 
$[m,f]$ is indeed a morphism from $\N^{(\hat{m}, \hat{f})}$ to $\N$ in $\GTSh$. 

It is clear that $\PR_{\N}(0, 1_{\wh{\F}_2}) = [0, 1_{\F_2}]$ for every 
$\N \in  \NFI_{\PB_3}(\B_3)$, i.e. the functor $\PR$
sends the identity morphisms of $\GTh^{gen}_{\NFI}$ to the identity morphisms 
of $\GTSh$. 

It remains to prove that, for all $(\hat{m}_1, \hat{f}_1),  (\hat{m}_2, \hat{f}_2) \in \GTh_{gen}$ 
and $\N \in \NFI_{\PB_3}(\B_3)$,
\begin{equation}
\label{PR-composition}
\PR_{\N}(\hat{m}_1, \hat{f}_1) \bullet \PR_{\K}(\hat{m}_2, \hat{f}_2) =
\PR_{\N}(\hat{m}, \hat{f}), 
\end{equation}
where $(\hat{m}, \hat{f}) = (\hat{m}_1, \hat{f}_1) \bullet  (\hat{m}_2, \hat{f}_2)$, 
$(\hat{m}_1, \hat{f}_1)$ is viewed as a morphism from $\K:= \N^{(\hat{m}_1, \hat{f}_1)}$ to $\N$
and $(\hat{m}_2, \hat{f}_2)$ is viewed as a morphism from $\K^{(\hat{m}_2, \hat{f}_2)}$ to $\K$. 

Let $(m_1, f_1)$ and $(m_2, f_2)$ be pairs that represent the $\GT$-shadows 
$\PR_{\N}(\hat{m}_1, \hat{f}_1)$ and $\PR_{\K}(\hat{m}_2, \hat{f}_2)$, respectively. 
Since the source of $[m_1, f_1]$, $\K$,  coincides with the target of $[m_2, f_2]$, 
the $\GT$-shadows $[m_1, f_1]$ and $[m_2, f_2]$ can be composed in 
this order $[m_1, f_1] \bullet [m_2, f_2]$ and $[m_1, f_1] \bullet [m_2, f_2]$
is an element of $\GTSh(\H, \N)$, where 
$\H := \K^{(\hat{m}_2, \hat{f}_2)}$.
Recall that $N_{\ord} = K_{\ord} = H_{\ord}$\,.

We need to prove that 
\begin{equation}
\label{m-bar-versus-m-hat}
m + N_{\ord} \ZZ = \hcP_{N_{\ord}} (\hat{m})
\end{equation}
and 
\begin{equation}
\label{fN-F2-versus-f-hat}
f \N_{\F_2} = \hcP_{\N_{\F_2}} (\hat{f}),
\end{equation}
where
$$
m : = 2m_1m_2 + m_1 +m_2, \qquad 
f := f_1 E_{m_1, f_1}(f_2). 
$$

While \eqref{m-bar-versus-m-hat} is obvious, identity \eqref{fN-F2-versus-f-hat}
requires some work.  

First, we observe that the diagram
\begin{equation}
\label{E-hat-versus-E}
\begin{tikzpicture}
\matrix (m) [matrix of math nodes, row sep=2.3em, column sep=5.5em]
{\wh{\F}_2  &  \wh{\F}_2 \\ 
\F_2 &  \F_2/\N_{\F_2}\\};
\path[->, font=\scriptsize]
(m-2-1) edge node[above] {$w \mapsto E_{m_1, f_1}(w)\N_{\F_2}$} (m-2-2)
edge node[left] {$j$} (m-1-1)
(m-1-1) edge node[above] {$E_{\hat{m}_1, \hat{f}_1}$} (m-1-2)
(m-1-2) edge node[right] {$\hcP_{\N_{\F_2}}$} (m-2-2);
\end{tikzpicture}
\end{equation}
commutes. 

Second, since $\hcP_{\K_{\F_2}} (\hat{f}_2) = \hcP_{\K_{\F_2}}(f_2)$, we have
\begin{equation}
\label{f2-versus-f2hat}
\hat{f}_2 = f_2 \hat{b}, 
\end{equation}
where\footnote{Just as in Proposition \ref{prop:KF2-hat-NF2-hat}, 
we identify $\wh{\K}_{\F_2}$ with $\hcP^{-1}_{\K_{\F_2}}(1_{\F_2/\K_{\F_2}})$.} 
$\hat{b} \in \wh{\K}_{\F_2}$. Combining this observation with equation \eqref{Emf-hat-NF2-hat} in 
Proposition \ref{prop:KF2-hat-NF2-hat} and commutativity of diagram \eqref{E-hat-versus-E}, 
we deduce that
\begin{equation}
\label{hcP-Emf-hat-Emf}
\hcP_{\N_{\F_2}} \big(E_{\hat{m}_1, \hat{f}_1}(\hat{f}_2)\big) = 
\hcP_{\N_{\F_2}} \big(E_{\hat{m}_1, \hat{f}_1}(f_2)\big) = 
E_{m_1, f_1} (f_2)\, \N_{\F_2}\,.  
\end{equation}

Therefore  
$$
\hcP_{\N_{\F_2}} \big(  \hat{f}_1 E_{\hat{m}_1, \hat{f}_1}(\hat{f}_2) \big) = 
\hcP_{\N_{\F_2}} (  \hat{f}_1) \, \hcP_{\N_{\F_2}} \big(E_{\hat{m}_1, \hat{f}_1}(\hat{f}_2)\big)=
f_1\, E_{m_1, f_1} (f_2)\, \N_{\F_2}\,.
$$
Thus identity \eqref{fN-F2-versus-f-hat} holds and equation \eqref{PR-composition} follows.
\end{proof}

In view of Corollary \ref{cor:genuine-iff} which is proved in the next section,
we call $\PR$ the \e{approximation functor}.

\section{The version of the Main Line functor for $\GTh_{gen}$}
\label{sec:ML}

Recall that, for every isolated object $\N$ of the groupoid $\GTSh$, $\GT(\N)= \GTSh(\N,\N)$. 
In particular, $\GT(\N)$ is a (finite) group.

Let us show that the assignment 
\begin{equation}
\label{ML-objects}
\ML(\N):=\GT(\N)
\end{equation}
can be upgraded to a functor $\ML$ from the poset $\NFI^{isolated}_{\PB_3}(\B_3)$
to the category of finite groups. 

For $\N, \H \in \NFI^{isolated}_{\PB_3}(\B_3)$, $\N \le \H$, we set 
\begin{equation}
\label{ML-morphisms}
\ML(\N \to \H) := \cR_{\N, \H}.
\end{equation}
Recall that, due to Remark \ref{rem:cR-homomorphism},  
the map $\cR_{\N, \H}: \GT(\N) \to \GT(\H)$ is a group homomorphism. 

It is obvious that, if $\N^{(3)} \le \N^{(2)} \le \N^{(1)}$, then
\begin{equation}
\label{we-get-a-functor}
\cR_{\N^{(2)}, \N^{(1)}} \circ  \cR_{\N^{(3)}, \N^{(2)}} =  \cR_{\N^{(3)}, \N^{(1)}}.
\end{equation}

Thus formulas \eqref{ML-objects}, \eqref{ML-morphisms} define a functor $\ML$ from 
the poset  $\NFI^{isolated}_{\PB_3}(\B_3)$ to the category of finite groups. 
We call $\ML$ the \e{Main Line functor}.

Our next goal is to show that the group $\GTh_{gen}$ is isomorphic 
to $\lim(\ML)$. For this purpose, we need to prove the following auxiliary statement: 

\begin{prop}  
\label{prop:we-have-N}
For every positive integer $K$, there exists $\N \in  \NFI^{isolated}_{\PB_3}(\B_3)$ 
such that $K | N_{\ord}$. Furthermore, for every $\H \in \NFI(\F_2)$, there exists
$\N \in  \NFI^{isolated}_{\PB_3}(\B_3)$, such that $\N_{\F_2} \le \H$. Finally, 
for every pair $(K, \H) \in \ZZ_{\ge 1} \times \NFI(\F_2)$, there exists 
$\N \in \NFI^{isolated}_{\PB_3}(\B_3)$ such that 
 $K | N_{\ord}$ and $\N_{\F_2} \le \H$. 
\end{prop} 
\begin{proof} The proof of the first statement of the proposition is straightforward, 
so we leave it to the reader. 

Since $\H$ is a finite index normal subgroup of $\F_2$, there exists a group 
homomorphism $\psi$ from $\F_2$ to a finite group $G$ such that 
$$
\ker(\psi) = \H.
$$ 

Clearly, the formulas 
\begin{equation}
\label{psi-tilde}
\ti{\psi}(x_{12}):= \psi(x), \qquad  
\ti{\psi}(x_{23}):= \psi(y),  \qquad \ti{\psi}(c) := 1_{G}
\end{equation}
define a group homomorphism $\ti{\psi}: \PB_3 \to G$.

In general, the subgroup $\ker(\ti{\psi})$ is not normal in $\B_3$. 
So we denote by $\ti{\N}$ the normal core of $\ker(\ti{\psi})$ in $\B_3$. 
It is clear that $\ti{\N} \in \NFI_{\PB_3}(\B_3)$ and $\ti{\N}_{\F_2} \le \ker(\psi)$. 

Let
$$
\N ~: =~ \bigcap_{\K \in \Ob(\GTSh_{\conn}(\ti{\N}))}\, \K. 
$$
Due to Proposition \ref{prop:N-diamond}, $\N$ is an isolated object of $\GTSh$.
Moreover, since $\N \le \ti{\N}$, we have $\N_{\F_2} \le \ker(\psi)$. 

The second statement of the proposition is proved. 

For $(K, \H) \in \ZZ_{\ge 1} \times \NFI(\F_2)$, 
there exist $\N^{(1)}, \N^{(2)} \in \NFI^{isolated}_{\PB_3}(\B_3)$ such that 
$K | N^{(1)}_{\ord}$ and $\N^{(2)}_{\F_2} \le \H$. Due to Proposition \ref{prop:cap-isolated},
$$
\N := \N^{(1)} \cap \N^{(2)}
$$
is an isolated object of $\GTSh$. Using the inclusions 
$\N \subset  \N^{(1)}$ and $\N \subset \N^{(2)}$, it is not hard to 
show that $K | N_{\ord}$ and $\N_{\F_2} \le \H$, respectively.

The proposition is proved. 
\end{proof}

\bigskip

We are now ready to construct an isomorphism of groups $\GTh_{gen} \iso \lim(\ML)$.
\begin{thm}  
\label{thm:lim-ML}
Let $(\hat{m}, \hat{f}) \in \GTh_{gen}$ and $\N \in \NFI^{isolated}_{\PB_3}(\B_3)$. 
The formula 
\begin{equation}
\label{Psi}
\Psi(\hat{m}, \hat{f})(\N) := \PR_{\N}(\hat{m}, \hat{f})
\end{equation}
defines an isomorphism of groups $\Psi: \GTh_{gen} \iso \lim(\ML)$. 
Moreover, $\Psi$ is a homeomorphism (of topological spaces).   
\end{thm}  
\begin{proof}
Since $\N$ is an isolated object of the groupoid $\GTSh$, 
$\N^{(\hat{m}, \hat{f})} = \N$ for every $(\hat{m}, \hat{f}) \in  \GTh_{gen}$. 
Furthermore, Theorem \ref{thm:GThgen-GTSh} implies that
the assignment 
$$
(\hat{m}, \hat{f}) \mapsto  \PR_{\N}(\hat{m}, \hat{f})
$$
is a group homomorphism from $\GTh_{gen}$ to the finite group $\GT(\N)= \GTSh(\N, \N)$. 

It is clear that, for every $\N, \H \in \NFI^{isolated}_{\PB_3}(\B_3)$, $\N \le \H$, we have 
$$
\cR_{\N, \H} \circ \PR_{\N}(\hat{m}, \hat{f}) = \PR_{\H}(\hat{m}, \hat{f}).
$$
Thus the formula in \eqref{Psi} indeed defines a group homomorphism $\Psi : \GTh_{gen} \to  \lim(\ML)$.

To prove the theorem, we will construct a map $\Te : \lim(\ML) \to \GTh_{gen}$
and show that 
\begin{itemize}

\item $\Te$ is the inverse of $\Psi$ and 

\item $\Te$ is a homeomorphism of topological spaces.  

\end{itemize}
 
Let $\hat{T} \in  \lim(\ML)$, $K \in \ZZ_{\ge 1}$ and $\H \in \NFI(\F_2)$.

Due to Proposition \ref{prop:we-have-N},  there exists $\N \in \NFI^{isolated}_{\PB_3}(\B_3)$ 
such that $K | N_{\ord}$ and $\N_{\F_2} \le \H$. Let $(m, f) \in \ZZ \times \F_2$ be a pair  
that represents the $\GT$-shadow $\hat{T}(\N)$. 
We set 
\begin{equation}
\label{m-hat-f-hat}
\hat{m}(K) := m + K \ZZ, \qquad \hat{f}(\H) := f \H.  
\end{equation}

Since $\hat{T}$ belongs to $\lim(\ML)$, the residue class $\hat{m}(K)$ and
the coset $\hat{f}(\H)$ do not depend on the choice of $\N \in \NFI^{isolated}_{\PB_3}(\B_3)$, 
and the formulas in \eqref{m-hat-f-hat} define $\hat{m} \in \Zhat$ and $\hat{f} \in \wh{\F}_2$.

The element $\hat{f}$ belongs to the topological closure of 
the commutator subgroup $[\wh{\F}_2, \wh{\F}_2]$ in $\wh{\F}_2$ due 
to these properties: 

\begin{itemize}

\item for every $\N \in \NFI^{isolated}_{\PB_3}(\B_3)$, 
$\hat{f}(\N_{\F_2})  \in [ \F_2/\N_{\F_2} ,  \F_2/\N_{\F_2}],$

\item the open subsets
$$
\hcP^{-1}_{\N_{\F_2}} (1_{\F_2/\N_{\F_2}}) \subset \wh{\F}_2, \qquad \N \in \NFI^{isolated}_{\PB_3}(\B_3)
$$
form a basis of neighborhoods of $1_{\wh{\F}_2}$ in $\wh{\F}_2$. 

\end{itemize}

Let us prove that the resulting pair $(\hat{m}, \hat{f}) \in \Zhat \times \wh{\F}_2$ satisfies 
hexagon relations \eqref{hexa1-hat} and \eqref{hexa11-hat}.  
For this purpose, we consider 
$\L \in \NFI(\B_3)$ and observe that $\L \cap \PB_3 \in  \NFI_{\PB_3}(\B_3)$. 
In general, $\L \cap \PB_3$ is not an isolated object of the groupoid 
$\GTSh$. However, due to Proposition \ref{prop:N-diamond}, the subgroup 
$\N := \big( \L \cap \PB_3 \big)^{\dia}$ does belong to $\NFI^{isolated}_{\PB_3}(\B_3)$. 
Moreover, since $\N \le \L \cap \PB_3$, $\N$ is a subgroup of $\L$. 

As above, let $(m, f) \in \ZZ \times \F_2$ be a pair  
that represents the $\GT$-shadow $\hat{T}(\N)$. For such a pair $(m,f)$, 
we have 
$$
\hat{m}(N_{\ord}) = m + N_{\ord} \ZZ, \qquad 
\hat{f}(\N_{\F_2}) = f \N_{\F_2}.
$$ 

Evaluating the left hand side 
(resp. the right hand side) of the first hexagon relation \eqref{hexa1-hat} at $\N$, we get 
the left hand side (resp. the right hand side) of the first hexagon relation 
\eqref{hexa1} for $(m,f)$. Thus 
\begin{equation}
\label{hexa1-hat-N}
\big( \si_1^{2 \hat{m} + 1}   \hat{f}^{-1} \si_2^{2 \hat{m} + 1}  \hat{f} \big) (\N) ~ = ~   
\big( \hat{f}^{-1} \si_1 \si_2\, x_{12}^{-\hat{m}} c^{\hat{m}} \big) (\N).
\end{equation}

Similarly, evaluating the left hand side 
(resp. the right hand side) of the second hexagon relation \eqref{hexa11-hat} at $\N$, we get 
the left hand side (resp. the right hand side) of the second hexagon relation 
\eqref{hexa11} for $(m, f)$. Thus 
\begin{equation}
\label{hexa11-hat-N}
\big( \hat{f}^{-1} \si_2^{2\hat{m}+1} \hat{f} \, \si_1^{2\hat{m}+1} \big) (\N)  
~ = ~ \big( \si_2 \si_1 x_{23}^{-\hat{m}} c^{\hat{m}} \, \hat{f} \big) (\N).
\end{equation} 

Since $\N \le \L$, identities \eqref{hexa1-hat-N} and \eqref{hexa11-hat-N} imply that 
$$
\big( \si_1^{2 \hat{m} + 1}   \hat{f}^{-1} \si_2^{2 \hat{m} + 1}  \hat{f} \big) (\L) ~ = ~   
\big( \hat{f}^{-1} \si_1 \si_2\, x_{12}^{-\hat{m}} c^{\hat{m}} \big) (\L).
$$
$$
\big( \hat{f}^{-1} \si_2^{2\hat{m}+1} \hat{f} \, \si_1^{2\hat{m}+1} \big) (\L)  
~ = ~ \big( \si_2 \si_1 x_{23}^{-\hat{m}} c^{\hat{m}} \, \hat{f} \big) (\L).
$$

We proved that the pair $(\hat{m}, \hat{f})$ belongs to 
$\Zhat \times [\wh{\F}_2, \wh{\F}_2]^{top.\,cl.}$ and satisfies 
hexagon relations \eqref{hexa1-hat} and \eqref{hexa11-hat}. 

Thus the assignment $\hat{T} \mapsto (\hat{m}, \hat{f})$ defines a map  
\begin{equation}
\label{Theta}
\Te : \lim(\ML) \to \GTh_{gen, mon}\,,
\end{equation}
where $\GTh_{gen, mon}$ is the monoid defined in Section \ref{sec:GTh-gen-mon-GTh-gen}
(see Proposition \ref{prop:GT-gen-mon-submonoid}). 

Let us prove that $\Te$ is a homomorphism of monoids. 
For this purpose, we consider $\N \in \NFI^{isolated}_{\PB_3}(\B_3)$,
$\hat{T}_1, \hat{T}_2 \in  \lim(\ML)$ and set 
\begin{equation}
\label{mf-mf-hat}
(\hat{m}_1, \hat{f}_1) := \Te(\hat{T}_1), \qquad 
(\hat{m}_2, \hat{f}_2) := \Te(\hat{T}_2), 
\end{equation}
\begin{equation}
\label{mf-hat}
\hat{m} : = 2 \hat{m}_1 \hat{m}_2 + \hat{m}_1 + \hat{m}_2, \qquad 
\hat{f} : = \hat{f}_1 E_{\hat{m}_1, \hat{f}_1} (\hat{f}_2).
\end{equation}

Let $(m_1, f_1) \in \ZZ \times \F_2$ (resp. $(m_2, f_2) \in \ZZ \times \F_2$) be a pair that represents 
the $\GT$-shadow $\hat{T}_1(\N) \in \GT(\N)$ (resp. the $\GT$-shadow $\hat{T}_2(\N) \in \GT(\N)$) and 
\begin{equation}
\label{mf}
m : = 2 m_1 m_2 + m_1 + m_2, \qquad 
f : = f_1 E_{m_1, f_1} (f_2),
\end{equation}
i.e. the pair $(m, f)$ represents the $\GT$-shadow $\hat{T}_1 \bullet \hat{T}_2(\N)$.

To prove the compatibility of $\Te$ with the multiplications in $\lim(\ML)$ and $\GTh_{gen, mon}$, we need to show that 
\begin{equation}
\label{hat-m-m}
\hat{m}(N_{\ord})  = m + N_{\ord} \ZZ\,.
\end{equation}
and
\begin{equation}
\label{hat-f-f}
\hat{f} (\N_{\F_2}) = f \N_{\F_2}. 
\end{equation}
Equation \eqref{hat-m-m} is clearly satisfied. 

As for \eqref{hat-f-f}, since $\hcP_{\N_{\F_2}} : \wh{\F}_2 \to \F_2/\N_{\F_2}$ is a group homomorphism
and $\hcP_{\N_{\F_2}} (\hat{f}_1) = f_1 \N_{\F_2}$, we need to show that 
\begin{equation}
\label{Emf-hat-Emf}
\hcP_{\N_{\F_2}} \big( E_{\hat{m}_1, \hat{f}_1} (\hat{f}_2)\big) ~ = ~  E_{m_1, f_1} (f_2)\, \N_{\F_2}\,.
\end{equation}
This identity was already established in a more general case in the proof of Theorem \ref{thm:GThgen-GTSh}
(see \eqref{hcP-Emf-hat-Emf}). 

It is easy to see that $\Te$ sends the identity element of the group
$\lim(\ML)$ to the identity element of the monoid $\GTh_{gen, mon}$. 


Since $\Te :  \lim(\ML) \to \GTh_{gen, mon}$ is a homomorphism of monoids 
and $\lim(\ML)$ is a group, $\Te(\lim(\ML))$ is a subset of invertible elements of 
the monoid $\GTh_{gen, mon}$. Thus $\Te$ is a group homomorphism 
from $\lim(\ML)$ to $\GTh_{gen}$.

It is clear that 
$$
\Te \circ \Psi = \id_{\GTh_{gen}} 
\qquad \txt{and} \qquad 
\Psi \circ \Te = \id_{\lim(\ML)}, 
$$
i.e. $\Te$ is indeed the inverse of $\Psi$. 

To prove the continuity of $\Te$, we consider it as the map 
from $\lim(\ML)$ to the topological space $\Zhat \times \wh{\F}_2$ and denote by 
$P_{\Zhat}$ (resp. $P_{\wh{\F}_2}$) the projection 
$\Zhat \times \wh{\F}_2 \to \Zhat$ (resp. $\Zhat \times \wh{\F}_2 \to \wh{\F}_2$). 
We need to show that the maps 
$P_{\Zhat} \circ \Te: \lim(\ML) \to \Zhat$ and $P_{\wh{\F}_2} \circ \Te: \lim(\ML)  \to  \wh{\F}_2$
are continuous. 

For a positive integer $K$, we choose $\N \in \NFI^{isolated}_{\PB_3}(\B_3)$
such that $K | N_{\ord}$. Since the map 
$$
\hcP_{K} \circ P_{\Zhat} \circ \Te : \lim(\ML) \to \ZZ/ K\ZZ
$$
factors through the continuous map $\lim(\ML) \to \ZZ/ N_{\ord} \ZZ$, the composition  
$\hcP_{K} \circ P_{\Zhat} \circ \Te$ is continuous. Hence the composition 
$P_{\Zhat} \circ \Te :  \lim(\ML) \to \Zhat$ is continuous. 

Similarly, for $\H \in \NFI(\F_2)$, we choose $\N \in \NFI^{isolated}_{\PB_3}(\B_3)$
such that $\N_{\F_2} \le \H$. Since the map 
$$
\hcP_{\H} \circ P_{\wh{\F}_2} \circ \Te : \lim(\ML) \to \F_2/\H
$$
factors through the continuous map $\lim(\ML) \to \F_2/ \N_{\F_2}$, the composition  
$\hcP_{\H} \circ P_{\wh{\F}_2} \circ \Te$ is continuous.  Hence the composition 
$P_{\wh{\F}_2} \circ \Te :  \lim(\ML) \to \wh{\F}_2$ is continuous. 

Since both maps $P_{\Zhat} \circ \Te :  \lim(\ML) \to \Zhat$ and 
$P_{\wh{\F}_2} \circ \Te :  \lim(\ML) \to \wh{\F}_2$ are continuous, 
so is the map $\Te:  \lim(\ML) \to  \Zhat \times  \wh{\F}_2$.

Now it is easy to see that $\Te:  \lim(\ML) \to  \GTh_{gen}$ is a homeomorphism. 
Indeed, $\Te$ is a continuous bijection from the compact topological space 
$\lim(\ML)$ to a Hausdorff space $\GTh_{gen}$. Thus $\Te$ (as well as $\Psi$)
is homeomorphism. 

Theorem \ref{thm:lim-ML} is proved. 
\end{proof}
\begin{remark}  
\label{rem:GTh-gen-topological-group}
As we mentioned in Remark \ref{rem:topology}, it is not obvious that 
$\GTh_{gen}$ is a topological group with respect to the subset topology coming 
from $\Zhat \times \wh{\F}_2$. However, since $\lim(\ML)$ is obviously a topological 
group, Theorem \ref{thm:lim-ML} implies that $\GTh_{gen}$ is indeed a topological group 
with respect to the subset topology coming from $\Zhat \times \wh{\F}_2$. 
\end{remark}
\begin{cor}  
\label{cor:genuine-iff}
Let $\N \in \NFI_{\PB_3}(\B_3)$. A $\GT$-shadow $[m, f] \in \GT(\N)$ is 
genuine if and only if $[m, f]$ belongs to the image of the map 
$$
\cR_{\K, \N} : \GT(\K) \to \GT(\N)
$$
for every $\K \in \NFI_{\N}(\B_3)$.  
\end{cor}  
\begin{proof}
If $[m, f] \in \GT(\N)$ is genuine, then $[m, f]$ obviously belongs to the image of 
the map $\cR_{\K, \N} : \GT(\K) \to \GT(\N)$ for every $\K \in \NFI_{\N}(\B_3)$. 

Thus it remains to prove the ``if'' implication. 

For $\K \in \NFI_{\N}(\B_3)$, we set 
$$
\cF(\K) := \cR_{\K, \N}^{-1}([m,f]) \subset \GT(\K). 
$$
Due to the given condition on $[m, f]$, the set $\cF(\K)$ is non-empty 
for every $\K \in \NFI_{\N}(\B_3)$. 

Property \eqref{we-get-a-functor} implies 
that the assignment $\K \mapsto \cF(\K)$ upgrades to a functor 
from the poset  $\NFI_{\N}(\B_3)$ to the category of finite sets 
(since $\GT(\K)$ is finite, so is $\cR_{\K, \N}^{-1}([m,f])$). 
Indeed, if $\H, \K \in \NFI_{\N}(\B_3)$ and $\H \le \K$, then 
$ \cR_{\H, \K}(\cF(\H)) \subset \cF(\K)$. So we set
$$
\cF(\H \to \K)  :=  \cR_{\H, \K} \big|_{\cF(\H)} : \cF(\H) \to  \cF(\K).  
$$

Since $\cF(\K)$ is a finite non-empty set for every $\K \in  \NFI_{\N}(\B_3)$, 
\cite[Proposition 1.1.4]{RZ-profinite} implies that $\lim(\cF)$ is non-empty. 

Taking an arbitrary element in $\lim(\cF)$ and evaluating it at elements of 
the poset $\NFI_{\N}(\B_3) \cap \NFI^{isolated}_{\PB_3}(\B_3)$, 
we get an element $(\hat{m}, \hat{f}) \in  \GTh_{gen} \cong \lim(\ML)$ such that 
$$
\PR_{\N}(\hat{m}, \hat{f}) = [m, f].  
$$

Thus the $\GT$-shadow $[m, f]$ is indeed genuine. 
\end{proof}

\subsection{Simplified hexagon relations in the profinite setting}
\label{sec:simple-hexagons}

In this section, we prove that 
\begin{prop}  
\label{prop:H-I-H-II}
The group $\GTh_{gen}$ (see Definition \ref{dfn:GTh-gen}) is isomorphic to  
the group $\GTh_0$ introduced in \cite[Section 0.1]{HS-fund-groups}.
\end{prop}  
\begin{proof}
According to \cite[Section 0.1]{HS-fund-groups}, $\GTh_0$ consists of elements
$(\hat{\la}, \hat{f}) \in \Zhat^\times \times [\wh{\F}_2, \wh{\F}_2]^{top.\, cl.}$ for which
the pair $(\hat{m}, \hat{f}) := \big((\hat{\la}-1)/2, \hat{f} \big) \in \Zhat \times [\wh{\F}_2, \wh{\F}_2]^{top.\, cl.}$
satisfies relations \eqref{H-I-hat}, \eqref{H-II-hat} and the endomorphism $E_{\hat{m}, \hat{f}}$ of $\wh{F}_2$ 
is invertible. In fact, the authors of \cite{HS-fund-groups} identify elements $(\hat{\la}, \hat{f})$ of $\GTh_0$ 
with the corresponding automorphisms $E_{\hat{m}, \hat{f}}$ of $\wh{F}_2$ and this is how they get 
the group structure on $\GTh_0$. 

Let us start with an element $(\hat{\la}, \hat{f}) \in \GTh_0$ and consider the corresponding pair
$$
(\hat{m}, \hat{f}) := 
\big((\hat{\la}-1)/2, \hat{f} \big) \in \Zhat \times [\wh{\F}_2, \wh{\F}_2]^{top.\, cl.}\,.
$$

Relations \eqref{H-I-hat}, \eqref{H-II-hat} imply that, for every $\N \in \NFI_{\PB_3}(\B_3)$, 
the pair 
$$
(m + N_{\ord}\ZZ, f \N_{\F_2}) : = \big(\hcP_{N_{\ord}}(\hat{m}), \hcP_{\N_{\F_2}}(\hat{f}) \big)
$$
satisfies relations \eqref{shexagon1} and \eqref{shexagon11}. In addition, we have 
$f \N_{\F_2} \in  [\F_2/\N_{\F_2}, \F_2/\N_{\F_2}]$. 

Thus Proposition \ref{prop:simple-hexa} implies that, for every $\N \in \NFI_{\PB_3}(\B_3)$, 
the pair $(m + N_{\ord}\ZZ, f \N_{\F_2}) : = \big(\hcP_{N_{\ord}}(\hat{m}), \hcP_{\N_{\F_2}}(\hat{f}) \big)$
satisfies relations \eqref{hexa1}, \eqref{hexa11}.  Since $\NFI_{\PB_3}(\B_3)$ is a cofinal subposet 
of $\NFI(\B_3)$, we conclude that the pair $(\hat{m}, \hat{f})$ satisfies hexagon relations 
\eqref{hexa1-hat} and \eqref{hexa11-hat}. 

Thus $(\hat{m}, \hat{f})$ belongs to the submonoid $\GTh_{gen, mon}$ and we need to show 
that the element $(\hat{m}, \hat{f})$ is invertible.

For this purpose we set
\begin{equation}
\label{formula-4-inverse}
\hat{k} := - (2\hat{m}+1)^{-1}\hat{m}, \qquad 
\hat{g} := E_{\hat{m}, \hat{f}}^{-1} (\hat{f}^{-1}).
\end{equation}

A direct computation shows that 
$$
(\hat{m}, \hat{f}) \bullet (\hat{k}, \hat{g}) = (0,1_{\wh{F}_2}).
$$
Therefore $E_{\hat{m}, \hat{f}} \circ E_{\hat{k}, \hat{g}} = \id_{\wh{F}_2}$ and hence 
\begin{equation}
\label{E-new-E-inverse}
E_{\hat{k}, \hat{g}} = E_{\hat{m}, \hat{f}}^{-1}
\end{equation}

Using \eqref{formula-4-inverse} and \eqref{E-new-E-inverse}, we get 
$$
2 \hat{k} \hat{m} + \hat{k} + \hat{m} = 0, 
$$
$$
\hat{g}\, E_{\hat{k}, \hat{g}} (\hat{f}) = \hat{g} \, E_{\hat{m}, \hat{f}}^{-1} (\hat{f}) = 
E_{\hat{m}, \hat{f}}^{-1} (\hat{f}^{-1}) \, E_{\hat{m}, \hat{f}}^{-1} (\hat{f})  = 
E_{\hat{m}, \hat{f}}^{-1} (1_{\wh{F}_2}) = 1_{\wh{F}_2}. 
$$
Thus the identity $(\hat{k}, \hat{g}) \bullet (\hat{m}, \hat{f}) = (0,1_{\wh{F}_2})$ is also satisfied
and the element $(\hat{m}, \hat{f})$ of the monoid $(\Zhat \times \wh{F}_2, \bullet) $ is indeed invertible.

Since $\hat{f} \in [\wh{\F}_2, \wh{\F}_2]^{top.\, cl.}$, the second equation in \eqref{formula-4-inverse}
and the continuity of the automorphism $ E_{\hat{m}, \hat{f}}^{-1}$
imply that $\hat{g} \in [\wh{\F}_2, \wh{\F}_2]^{top.\, cl.}$. Thus it remains to prove that 
the pair $(\hat{k}, \hat{g})$ satisfies hexagon relations \eqref{hexa1-hat}, \eqref{hexa11-hat}. 

Let us rewrite the right hand side of \eqref{hexa11-hat} for $(\hat{k}, \hat{g})$ as follows: 
$$
\si_2 \si_1 x_{23}^{-\hat{k}} c^{\hat{k}} \hat{g} = \D \si_2^{-(2 \hat{k} + 1)}  c^{\hat{k}} \hat{g}.
$$ 
Applying $T_{\hat{m}, \hat{f}}$ to the right hand side of \eqref{hexa11-hat} for $(\hat{k}, \hat{g})$ 
and using \eqref{Tmf-c},  \eqref{Tmf-hat-F2}, \eqref{Tmf-Delta1}, we get 
$$
T_{\hat{m}, \hat{f}} (\si_2 \si_1 x_{23}^{-\hat{k}} c^{\hat{k}} \hat{g}) = 
T_{\hat{m}, \hat{f}}  ( \D \si_2^{-(2 \hat{k} + 1)}  c^{\hat{k}} \hat{g} ) = 
$$
$$
\D c^{\hat{m}} \hat{f} \, \hat{f}^{-1} \si_2^{-(2 \hat{m} + 1)(2 \hat{k} + 1)} \hat{f}
c^{(2\hat{m} + 1) \hat{k}} \hat{f}^{-1} = \D \si_2^{-1} = \si_2 \si_1.
$$

Thus 
\begin{equation}
\label{Tmf-RHS}
T_{\hat{m}, \hat{f}} (\si_2 \si_1 x_{23}^{-\hat{k}} c^{\hat{k}} \hat{g}) =  \si_2 \si_1\,.
\end{equation}

Applying $T_{\hat{m}, \hat{f}}$ to the left hand side of \eqref{hexa11-hat} for $(\hat{k}, \hat{g})$, 
we get 
\begin{equation}
\label{Tmf-LHS}
T_{\hat{m}, \hat{f}} (\hat{g}^{-1} \si_2^{2\hat{k} + 1} \hat{g} \, \si_1^{2\hat{k} + 1}) = 
E_{\hat{m}, \hat{f}}(\hat{g})^{-1} 
\hat{f}^{-1} \si_2^{(2\hat{m}+1)(2\hat{k} + 1)} \hat{f} E_{\hat{m}, \hat{f}} (\hat{g}) \,
\si_1^{(2\hat{m}+1)(2\hat{k} + 1)} = \si_2 \si_1\,.
\end{equation}

Since $T_{\hat{m}, \hat{f}}$ is an automorphism of $\wh{\B}_3$, identities 
\eqref{Tmf-RHS} and \eqref{Tmf-LHS} imply that 
$$
\hat{g}^{-1} \si_2^{2\hat{k} + 1} \hat{g} \, \si_1^{2\hat{k} + 1} = 
\si_2 \si_1 x_{23}^{-\hat{k}} c^{\hat{k}} \hat{g}.
$$
Thus the pair $(\hat{k}, \hat{g})$ satisfies \eqref{hexa11-hat}. 

Using the similar argument, one can show that the pair $(\hat{k}, \hat{g})$ also 
satisfies \eqref{hexa1-hat}.

We proved that the pair $(\hat{m}, \hat{f})$ belongs to the group $\GTh_{gen}$.   

Let $(\hat{m}, \hat{f}) \in \GTh_{gen}$, i.e. $(\hat{m}, \hat{f})$ is an invertible element 
of the monoid $\GTh_{gen, mon}$. Let us prove that the pair 
$$
(\hat{\la}, \hat{f}), \qquad \hat{\la} := 2 \hat{m} + 1 
$$ 
belongs to the group $\GTh_0$. 

Relations  \eqref{hexa1-hat} and  \eqref{hexa11-hat} imply that, 
for every $\N \in \NFI_{\PB_3}(\B_3)$, the pair 
\begin{equation}
\label{pair-mod-N}
(m + N_{\ord}\ZZ, f \N_{\F_2}) : = \big(\hcP_{N_{\ord}}(\hat{m}), \hcP_{\N_{\F_2}}(\hat{f}) \big)
\end{equation}
satisfies hexagon relations \eqref{hexa1} and \eqref{hexa11} modulo $\N$. 
In addition, 
$
f \N_{\F_2} \in  [\F_2/\N_{\F_2}, \F_2/\N_{\F_2}].
$

Thus Proposition \ref{prop:simple-hexa} implies that, 
for every $\N \in \NFI_{\PB_3}(\B_3)$, the pair in \eqref{pair-mod-N}
satisfies relations \eqref{shexagon1}, \eqref{shexagon11}. 
 
Due to Proposition \ref{prop:we-have-N}, for every $\H \in \NFI(\F_2)$, there exists
$\N \in \NFI_{\PB_3}(\B_3)$ such that $\N_{\F_2} \le \H$. Thus the above observation 
about \eqref{shexagon1} and \eqref{shexagon11} implies that the pair 
$(\hat{m}, \hat{f})$ satisfies relations \eqref{H-I-hat} and \eqref{H-II-hat}. 
 
Since $\hat{f} \in [\wh{\F}_2, \wh{\F}_2]^{top.\, cl.}$ and $\hat{\la} = 2\hat{m}+1$ is a unit in the ring
$\Zhat$ (see Remarks \ref{rem:cyclotomic}, \ref{rem:virtual-cyclotomic}), it remains to show that 
the endomorphism $E_{\hat{m}, \hat{f}}$ is invertible. This is an obvious consequence of 
the second statement of Proposition \ref{prop:GT-gen-mon-submonoid}. Indeed, if $\vf: M \to \ti{M}$
is a homomorphism of monoids, the restriction of $\vf$ to the group $M^{\times}$ of invertible 
elements of $M$ gives us a group homomorphism $M^{\times} \to \ti{M}^{\times}$.  
 
We established a bijection between the set $\GTh_0$ (defined in \cite[Section 0.1]{HS-fund-groups})
and the set $\GTh_{gen}$. It remains to prove that this bijection is compatible with the group structures on 
$\GTh_0$ and $\GTh_{gen}$. Since the group structure on $\GTh_0$ is obtained by identifying 
elements $(\hat{\la}, \hat{f})$ of $\GTh_0$ with the corresponding automorphisms $E_{\hat{m}, \hat{f}}$ 
of $\wh{\F}_2$, the desired property follows from the second statement of Proposition 
\ref{prop:GT-gen-mon-submonoid}. 

Proposition \ref{prop:H-I-H-II} is proved. 
\end{proof}

\begin{remark}  
\label{rem:simple-hexagons}
Relations \eqref{H-I-hat} and \eqref{H-II-hat} may be interpreted as cocycle conditions 
and this interpretation was explored successfully in \cite{SashaLeilaCohomological}. 
\end{remark}

\appendix

\section{Selected statements related to profinite groups}
\label{app:profinite}

In this appendix, we prove several statements related to profinite groups.
These statements are often used in articles about the profinite version of the 
Grothendieck-Teichmueller group. However, it is hard to find proofs of these statements 
in the literature.  

Let $J$ be a directed poset and $\cF$ be a functor from $J$ to the category of finite 
groups. For $k_1, k_2 \in J$, $k_1 \le k_2$ we set $\te_{k_1 , k_2} := \cF(k_1 \to k_2)$.
 
It is convenient to identify elements of the product 
\begin{equation}
\label{product}
\prod_{k \in J} \cF(k)
\end{equation}
with functions 
\begin{equation}
\label{fun-J-bigsqcup}
f : J \, \to \,  \bigsqcup_{k \in J}  \cF(k)
\end{equation}
such that $f(k) \in \cF(k)$, $\forall ~ k \in J$.

Then $\lim(\cF)$ consists of functions \eqref{fun-J-bigsqcup} such that 
\begin{itemize}

\item $f(k) \in \cF(k)$, $\forall ~ k \in J$ and 

\item $\te_{k_1, k_2}(f(k_1)) = f(k_2)$,  $\forall~~ k_1, k_2 \in J$, $k_1 \le k_2$. 

\end{itemize}

For $k \in J$, $\eta_k$ denotes the standard projection from 
$\lim(\cF)$ to $\cF(k)$, i.e. 
$$
\eta_k(f) := f(k). 
$$

We consider the product space \eqref{product} with the standard product 
topology and we equip $\lim(\cF)$ with the corresponding subset topology. 
Let us also recall \cite[Proposition 1.1.3]{RZ-profinite} that, as the topological space, 
$\lim(\cF)$ is compact and Hausdorff. It is known \cite[Section 1.1]{RZ-profinite} that
every profinite group is $\lim(\cF)$ for a functor $\cF$ from a directed poset 
to the category of finite groups. 
 
For every group $G$, the poset $\NFI(G)$ is clearly directed and 
the assignments
$$
\N \mapsto G/\N, \qquad 
\te_{\K, \N} := \cP_{\K, \N}: G/ \K \to G/ \N,  
\qquad \K, \N \in \NFI(G), ~~ \K \le \N
$$ 
define a functor $\cF_{G}$ from $\NFI(G)$ to the category of finite groups. 
The profinite completion $\wh{G}$ of $G$ is the limit $\lim(\cF_{G})$ of this functor.

As we mentioned above, it is convenient to identify elements $\hat{g}$ of $\wh{G}$ with 
functions
$$
\hat{g}: \NFI(G) \,\to\, \bigsqcup_{\N \in \NFI(G)} G/\N
$$
such that 
\begin{itemize}

\item $\hat{g}(\N) \in G/\N$, $\forall~~ \N \in \NFI(G)$ and 

\item $\cP_{\K, \N} \big(\hat{g}(\K) \big) = \hat{g}(\N),$
$\forall$ $\K, \N \in \NFI(G)$, $\K \le \N$.

\end{itemize}

In this set-up, $\eta_{\N} := \hcP_{\N}$.

We denote by $j$ the standard group homomorphism $G \to \wh{G}$ defined 
by the formula 
$$
j(g)(\N) := g \N, \qquad \N \in \NFI(G).
$$
Recall \cite[Lemma 1.1.7]{RZ-profinite} that, for every group $G$, the subgroup $j(G)$ is dense 
in $\wh{G}$. Moreover, the homomorphism $j: G \to \wh{G}$ is injective 
if and only if the group $G$ is residually finite.  

\begin{lem}  
\label{lem:extend-uniquely}
Let $G$ be a group and $j$ be the standard homomorphism $G \to \wh{G}$.
For every group homomorphism $\vf$ from $G$ to a profinite group $H$, there exists 
a unique continuous group homomorphism 
$$
\hat{\vf} : \wh{G} \to H
$$
such that $\hat{\vf} \circ j = \vf$. 
\end{lem}  
\begin{proof}
Since $H$ is a profinite group, there exists a directed poset $J$ and a functor 
$\cF$ from $J$ to the category of finite groups such that 
$H = \lim(\cF)$. For $k \in J$, we denote by $\eta_k$ the standard continuous
group homomorphism from $H$ to $\cF(k)$.

For every $k \in J$, $\eta_k \circ \vf$ is a homomorphism from $G$ to 
the finite group $\cF(k)$. Hence $\ker(\eta_k \circ \vf)$ is a finite index normal 
subgroup of $G$. We denote this subgroup by $\N_k$,  
$$
\N_k := \ker \big( G \overset{\eta_k \circ \vf}{~\tto~} \cF(k) \big). 
$$

It is easy to see that the formula 
\begin{equation}
\label{vf-k}
\vf_k(g \N_k) :=  \eta_k \circ \vf(g) 
\end{equation}
defines a group homomorphism from the finite group $G/\N_k$ to the finite 
group $\cF(k)$.

Let us also observe that, if $k_1, k_2 \in J$ and $k_1 \le k_2$ then 
$\N_{k_1} \le \N_{k_2}$ and the diagram 
\begin{equation}
\label{vf-k-diagram}
\begin{tikzpicture}
\matrix (m) [matrix of math nodes, row sep=2.3em, column sep=2.3em]
{G/\N_{k_1}   &  \cF(k_1) \\ 
G/\N_{k_2} &  \cF(k_2)\\};
\path[->, font=\scriptsize]
(m-1-1) edge node[above] {$\vf_{k_1}$} (m-1-2)
edge node[left] {$\cP_{\N_{k_1} , \N_{k_2} }$} (m-2-1) 
(m-2-1) edge node[above] {$\vf_{k_2}$} (m-2-2)
(m-1-2) edge node[right] {$\te_{k_1 , k_2}$} (m-2-2);
\end{tikzpicture}
\end{equation}
commutes. Here $\te_{k_1 , k_2} := \cF(k_1 \to k_2)$. 

We claim that the formula 
\begin{equation}
\label{vf-hat-dfn}
(\hat{\vf}(\hat{g}))(k) := \vf_{k}\big( \hat{g} (\N_k)  \big), \qquad k \in J
\end{equation}
defines a continuous group homomorphism $\hat{\vf}$ from $\wh{G}$ to $H$. 

Indeed, it is obvious that, for every $k \in J$ and every $\hat{g} \in \wh{G}$, 
$(\hat{\vf}(\hat{g}))(k) \in \cF(k)$. Thus $\hat{\vf}(\hat{g})$ belongs to the product 
$$
\prod_{k \in J} \cF(k)\,.
$$

The commutativity of the diagram in \eqref{vf-k-diagram} implies that 
$\hat{\vf}(\hat{g})$ satisfies the condition
$$
\te_{k_1, k_2} \big( (\hat{\vf}(\hat{g})) (k_1) \big) = (\hat{\vf}(\hat{g}))(k_2)
$$
whenever $k_1 \le k_2$. Thus $\hat{\vf}(\hat{g})$ belongs to 
$\displaystyle H \subset \prod_{k \in J} \cF(k)$. 

It is easy to see that $\hat{\vf}$ is indeed a group homomorphism $\wh{G} \to H$. 

Equation \eqref{vf-hat-dfn} implies that 
$$
\eta_k \circ \hat{\vf} = \vf_k \circ \hcP_{\N_k}\,.
$$
Hence the composition $\eta_k \circ \hat{\vf}$ is continuous for every $k \in J$. 

Thus we proved that equation \eqref{vf-hat-dfn} indeed defines a continuous 
group homomorphism from $\wh{G}$ to $H$.

Using \eqref{vf-k}, we see that, for every $k \in J$ and $g \in G$, 
we have 
$$
\big( \hat{\vf}\circ j (g)\big)(k) = \vf_k (g \N_k) = \eta_k (\vf(g)). 
$$
Thus $\hat{\vf} \circ j = \vf$.

Let $\psi: \wh{G} \to H$ be a continuous group homomorphism such 
that $\psi \circ j = \vf$. Since $\hat{\vf} \circ j = \vf$, we have 
\begin{equation}
\label{psi-hat-vf-agree}
\psi \big|_{j(G)} = \wh{\vf}\big|_{j(G)}.
\end{equation}

Since $j(G)$ is dense in $\wh{G}$ and $H$ is Hausdorff, identity  
\eqref{psi-hat-vf-agree} implies that $\psi = \hat{\vf}$. Thus the uniqueness of 
$\hat{\vf}$ is established and the lemma is proved.
\end{proof}

\begin{cor}  
\label{cor:extend-uniquely}
Let $G$, $H$ be groups and $j$ be the standard homomorphism $G \to \wh{G}$.
For every group homomorphism $\vf :G  \to \wh{H}$, there exists 
a unique continuous group homomorphism 
$$
\hat{\vf} : \wh{G} \to \wh{H}
$$
such that $\hat{\vf} \circ j = \vf$. If $\ga$ is an automorphism of $G$ then 
$\wh{j \circ \ga}$ is a continuous automorphism of $\wh{G}$. 
\end{cor}

\begin{proof}
The first statement of the corollary follows Lemma \ref{lem:extend-uniquely}.

Let $\ga \in \Aut(G)$ and $\ka := \ga^{-1}$. By abuse of notation, we denote 
by $\hat{\ga}$ (resp. $\hat{\ka}$) the continuous group homomorphism $\wh{G} \to \wh{G}$ 
corresponding to $j \circ \ga$ (resp. to  $j \circ \ka$). 

For $\hat{\ga}$ and $\hat{\ka}$, we have 
$$
\hat{\ga} \circ j = j \circ \ga, \qquad \hat{\ka} \circ j = j \circ \ka.
$$
Using these identities, we get 
$$
\hat{\ga} \circ \hat{\ka} \circ j = \hat{\ga} \circ j \circ \ka = j \circ \ga \circ \ka = j
$$ 
and 
$$
\hat{\ka} \circ \hat{\ga} \circ j = \hat{\ka} \circ j \circ \ga = j \circ \ka \circ \ga = j.
$$ 

Since $\hat{\ga} \circ \hat{\ka} \big|_{j(G)} = \id \big|_{j(G)}$,  
$\hat{\ka} \circ \hat{\ga} \big|_{j(G)} = \id \big|_{j(G)}$, $j(G)$ is dense in $\wh{G}$ 
and $\wh{G}$ is Hausdorff, we conclude that 
$$
\hat{\ga} \circ \hat{\ka}  = \id_{\wh{G}}, \qquad 
\hat{\ka} \circ \hat{\ga}  = \id_{\wh{G}}.  
$$ 
Thus $\hat{\ga}$ is invertible and $\hat{\ka} = \hat{\ga}^{-1}$. 
\end{proof}

\bigskip

Let us prove that 
\begin{prop}  
\label{prop:hat-N-kernel}
For every $\N \in \NFI(G)$, the kernel of the homomorphism 
$\hcP_{\N} : \wh{G} \to G/\N$ is isomorphic to the profinite completion 
$\wh{\N}$ of $\N$.  
\end{prop}  
\begin{proof}
For every $L \in \NFI(\N)$, the normal core $\Core_G(L)$ of $L$ in $G$ is an element of 
$\NFI_{\N}(G)$. Therefore the subposet $\NFI_{\N}(G)$ of $\NFI(\N)$ is cofinal and 
hence the limit of the functor 
\begin{equation}
\label{functor-from-NFI-N-G}
\H ~~\mapsto~~ \N / \H
\end{equation}
from $\NFI_{\N}(G)$ to the category of finite groups is isomorphic 
to $\wh{\N}$ (see \cite[Lemma 1.1.9]{RZ-profinite}).

Let 
$$
\K : = \ker\big(  \wh{G} \overset{\hcP_{\N}}{\tto} G/\N \big).
$$

For every $\H \in \NFI_{\N}(G)$, the restriction of the continuous homomorphism 
$\hcP_{\H} :  \wh{G} \to G/\H$ gives us a continuous homomorphism 
$$
\hcP_{\H} \big|_{\K} : \K \to \N / \H. 
$$ 
Moreover, for all $\H_1, \H_2 \in \NFI_{\N}(G)$ with $\H_1 \le \H_2$, the diagram 
$$
\begin{tikzpicture}
\matrix (m) [matrix of math nodes, row sep=1.8em, column sep=1.8em]
{ ~  & \K &  ~ \\
\N / \H_1 & ~ &  \N / \H_2\\};
\path[->, font=\scriptsize]
(m-1-2) edge node[above] {$\hcP_{\H_1}~~~$} (m-2-1)
edge node[above] {$~~~\hcP_{\H_2}$} (m-2-3)
(m-2-1)  edge node[above] {$\cP_{\H_1, \H_2}$}  (m-2-3);
\end{tikzpicture}
$$
commutes. 

Hence we get a continuous group homomorphism $\ga: \K \to \wh{\N}$, 
where $\wh{\N}$ is identified with the limit of functor  
\eqref{functor-from-NFI-N-G}. It is not hard to see that $\ga$ is a bijection.
Since $\K$ is compact ($\K$ is a closed subset of the compact space $\wh{G}$)
and $\ga$ is a continuous bijection from a compact space $\K$ to the Hausdorff 
space $\wh{\N}$, $\ga$ is a homeomorphism.  
Since $\ga$ is also an isomorphism of groups, we proved that the topological groups 
$\K$ and $\wh{\N}$ are isomorphic.
\end{proof}

\bigskip

\noindent\textsc{Department of Mathematics,
Temple University, \\
Wachman Hall Rm. 638\\
1805 N. Broad St.,\\
Philadelphia PA, 19122 USA \\
\emph{E-mail address:} {\bf vald@temple.edu}}

\bigskip
\bigskip

\noindent\textsc{School of Mathematics,\\
Georgia Institute of Technology \\
686 Cherry Street \\
Atlanta, GA 30332-0160 USA \\
\emph{E-mail address:} {\bf jguynee@gatech.edu}}

\end{document}